\title{A Driven Tagged Particle in Symmetric Exclusion Processes with Removals}
\author{Zhe Wang\thanks{
		Courant Institute}        
}
\date{Oct 8, 2019}

\documentclass[11pt,a4paper]{article}

\usepackage[english]{babel}
\usepackage[margin=1.0 in]{geometry}
\usepackage[latin1]{inputenc}
\usepackage{amsmath}
\usepackage{amsfonts}
\usepackage{amssymb}
\usepackage{fourier}
\usepackage{graphicx}
\usepackage{tikz}
\usepackage[inline]{enumitem}

\newtheorem{theorem}{Theorem}[section]
\newtheorem{corollary}{Corollary}[section]
\newtheorem{lemma}{Lemma}[section]

\newtheorem{remark}{Remark}

\numberwithin{equation}{section}

\newenvironment{proof}{{\sc Proof}}{~\hfill $\square$}

\def\mathbi#1{\textbf{\em #1}}
\def\abs#1{\left\vert #1 \right\vert}

\begin{document}
	\newpage
	\maketitle
	
	\begin{abstract}
			We consider a driven tagged particle in a symmetric exclusion process on $\mathbb{Z}$ with a removal rule. In this process, untagged particles are removed once they jump to the left of \textcolor{black}{the} tagged particle. We investigate the behavior of the displacement of the tagged particle and prove limit theorems of it: an (annealed) law of large numbers, a central limit theorem, and a large deviation principle. \textcolor{black}{We also characterize a class of ergodic measures for the environment process. Our approach is based on analyzing two auxiliary processes with associated martingales and a regenerative structure.} 
	\end{abstract}

\section{Introduction}
The symmetric exclusion process (SEP) on the lattice $\mathbb{Z}^d$ with a driven tagged particle can be described as: a collection of red particles and a tagged green particle performing continuous random walks on $\mathbb{Z}^d$ with an exclusion rule. There is at most one particle at each site. Particles have independent exponential clocks: the rate for a red particle is $\lambda=\sum_z p(z)$ and for the tagged particle is $\beta=\sum_z q(z)$. When its clock rings, a particle at site $x$ jumps to a vacant site $x+z$ with probability $\frac{p(z)}{\lambda}$ or $\frac{q(z)}{\beta}$ depending on its color, and the jump is suppressed if the site $x+z$ is occupied. \textcolor{black}{The word "symmetric" stands for the case when $p(z)=p(-z)$ for all $z$. When $q(\cdot)$ is different from $p(\cdot)$, we say the tagged particle is a driven particle. In this article, we add a removal rule that red particles are removed once they jumps to the left of the tagged particle. We will denote by $X_t$ the displacement $X_t$ of the tagged particle, and study the long term behavior of $X_t$.}

\textcolor{black}{This model is a variant of a classical problem: the SEP with a tagged particle. In the classical problem, $p(\cdot)=q(\cdot)$ is symmetric, and there are no removals of red particles. Limit theorems of $X_t$ has been studied \cite{Ar, J, KV, Sa}. A starting point is to consider the environment process $\xi_t$ viewed from the tagged particle, and write $X_t$ as a sum of a martingale and an additive functional in terms of $\xi_t$. The Bernoulli measures $\mu_\rho$ with parameters $\rho$ are reversible and ergodic for $\xi_t$. As a consequence, one can obtain a law of large numbers for $X_t$ \cite{Sa}. The fluctuation of $X_t$ in equilibrium is subdiffusive with a scale $t^{\frac{1}{2}}$ (the variance grows in $t^{\frac{1}{2}}$, and this corresponds to the slow-down phenomenon) when ${p(\cdot)=q(\cdot)}$ is nearest-neighbor in dimension $d=1$\cite{Ar}, and diffusive in other symmetric, finite-range cases \cite{KV}. A powerful method, proposed by Kipnis and Varadhan \cite{KV}, is to study the additive functionals of reversible Markov processes by martingale approximations. This method has also been extended to asymmetric cases \cite{SVY,Va}. For long-range jump rates $p(\cdot)=q(\cdot)$, there is a different scaling limit characterized by a Levy process \cite{J}.} 

%

The nearest-neighbor case in dimension $d=1$ is special. Particles are trapped and orders are preserved. The displacement $X_t$ can be considered jointly with the current through the bond $(0,1)$. On the other hand, we can construct the nearest-neighbor SEP with two other processes: the stirring process and the zero-range process. The former enables us to see negative correlations in the nearest-neighbor SEP \cite{Ar}, and the latter enables us to apply hydrodynamic limit results of the zero-range process to study the nearest-neighbor SEP. \textcolor{black}{Under the scale $t^{\frac{1}{2}}$, it's possible to obtain a nonequilibrium central limit theorem for $X_t$ by proving a joint central limit theorem for the current and density field \cite{JL}, and also to show a large deviation principle for the current and $X_t$ jointly \cite{SV}, see \cite{SV} for reviews. In fact, applying similar ideas to study the driven tagged particle case also works when $d=1$ and $p(\cdot), q(\cdot)$ are nearest-neighbor. In \cite{LOV} Landim, Olla, and Volchan showed that $X_t$ grows as $t^{\frac{1}{2}}$ and there is an Einstein relation for $X_t$ under scale $t^{\frac{1}{2}}$. They conjectured
	that $X_t$ grows linearly in t when the drift
	$\sum z \cdot q(z) \neq 0$ and $p(\cdot)$ is not nearest-neighbor in dimension $d=1$ or general in $d\geq 2$.}

	\textcolor{black}{This conjecture is partially proved when $d\geq 3$ but it remains open for $d\leq 2$. A major technical difficulty is to characterize some invariant measure for the environment process $\xi_t$. Due to asymmetry from different $p(\cdot)$ and $q(\cdot)$, it seems impossible to compute invariant measures explicitly other than trivial Bernoulli measures $\mu_\rho$, $\rho=0,1$. It is also unclear whether there are different invariant measures corresponding to different values of $\rho$. When $d\geq 3$, this difficulty can be overcome by comparing some invariant measure with the Bernoulli measure using large deviation estimates for the SEP and transient estimates \cite{Lo}. This is a perturbation argument. In \cite{KO}, a variant model of \cite{LOV} was considered in dimension $d=1$ with another perturbation argument. In this variant model, there are also annihilation and creation of red particles. A full expansion of an ergodic measure for $\xi_t$ is possible from the spectral gape property of the unperturbed system. With this expansion, one can show a law of large numbers and Einstein relation for $X_t$.}

\textcolor{black}{In this article, we study another variant of the SEP with a driven tagged particle \cite{LOV} in dimension $d=1$ when a removal rule is added, i.e. red particles are removed once they jump to the left of the tagged particle. We also assume that $q(\cdot)$ has only jumps towards right of size 1, $p(\cdot)$ is symmetric with jump sizes up to 2. We will obtain limit theorems of $X_t$ under the scale $t$ with explicit formulas (Theorem \ref{Thm: LLN for X_t}, \ref{Thm: CLT for X_t}, \ref{thm: LDP}). The main idea is to take the point of view of the tagged particle. We will construct an auxiliary process and study its environment process $\eta_t$ viewed from the tagged particle. Remarkably, there is an (implicit) ergodic measure $\nu_e$ for $\eta_t$ and we can find its marginal distribution at site 1 explicitly. This ergodic measure is central in our analysis. Firstly, we derive the law of large numbers of $X_t$ explicitly from the knowledge of the marginal distribution. Secondly, the rate function for large deviations of $X_t$ can be interpreted as the cost for perturbing the dynamics of $\eta_t$ near site one. The cost is in terms of some quantities of $\eta_t$ at site one, with asymptotic behaviors described by the ergodic measure for the perturbed dynamics. By pairing ergodic measures with perturbations in a specific way, we can derive the rate function in a closed form. Thirdly, the ergodic measure always gives us a regenerative structure for $\eta_t$ characterized by a regeneration time $\tilde{\tau}$. Moments of $\tilde{\tau}$ are hard to estimate but those of a variant are easier. We can look at a second auxiliary process and consider a regenerative time $\tau$ similar to $\tilde{\tau}$. Under an additional drift condition, we can compute moment generating functions of regeneration time $\tau$ and $X_\tau$, and deduce a central limit theorem for $X_t$ from them. Lastly, we can also characterize ergodic measures for $\xi_t$ from $\nu_e$, and these measures correspond to different $\rho$, see Remark \ref{rm: ergodic measure}.}
	
	\textcolor{black}{Due to the removal of the red particles, the tagged particle is the left-most particle. We can consider this model in two contexts: random walks in dynamical random environments (RWDRE), and interacting particle systems with moving boundaries. We briefly review some related works.
}	

\textcolor{black}{The point of view of the particle has been successful in showing limit theorems in several RWDRE models. To take this point of view, we need to show ergodicity or mixing properties of the environment process. Transferring mixing properties from the random environment to environment process (these two are different) is possible in some problems. For example, we refer to the models \cite{KO,Lo} mentioned earlier, \cite{A,RV} and the references therein. This will not be the approach in our model for two reasons. Firstly, the 1-D SEP environment has slowly decaying space-time correlations \cite{HS} (or see Lemma 2 \cite{Lo}). The density of red particles in a large box can stay close to $0$ or $1$ for a long time, which may cause the tagged particle to move at different speeds for different time intervals. Secondly, if we take the point of view of the environment, the tagged particle (or boundary) is moving towards right, and the environment has only a trivial invariant measure.}

	\textcolor{black}{An alternative approach is to show a regenerative structure for the environment process. This approach works in some front propagation models, such as \cite{BR,CQR07,CQR09,JMR}. In these models, two types of particles follow symmetric exclusion process or symmetric random walks with rates $D_A$ and $D_B$ on $\mathbb{Z}$. Type B particles are converted to type A particles up on contact or according to certain rule. The front is the right-most type A particle. In \cite{BR,CQR07,CQR09,JMR}, regenerative structures are observed for the environment process viewed from the front with good tail estimates, from which it follows that there is a unique invariant measure. A law of large numbers and a central limit theorem for the front also follow. This is very similar to our model. Briefly speaking, if the front moves ballistically while most particles move diffusively, one  expects regeneration times after which the trajectory of the front is decoupled from those of the particles around it. In our model, we can achieve this by showing that the average number of particles which affect the tagged particle has a finite mean, while the other particles are distributed according to independent Bernoulli random variables (see Lemma \ref{lm: finite N_b, N_p}). In fact, these front propagation models are believed to be microscopic versions of reaction-diffusion equations, one of which is the F-KPP equation by Fisher, and Kolmogorov,
		Petrovsky and Piscounov. It is known that they exhibit traveling wave solutions. We refer to \cite{BR,CQR07,KRS} and the references therein. Meanwhile,  
		fronts or boundaries can have different scaling limits due to different local mechanisms around the boundaries. For example, one can derive hydrodynamic limits under diffusive scaling for models \cite{LOV}, discrete atlas model \cite{HJV}, Stefan's melting and freezing
		problem \cite{LV}, and free boundary problems \cite{CdGP,DFP}, and they often correspond to Stefan problems \cite{CdGP,DFP,LOV,LV}.} 
	
	\textcolor{black}{Besides front propagation problems, regenerative structures have also been applied to random walks on interacting particle systems recently, such as \cite{ASV,dHdS,HdFSVST,HS,MV}. The model, a random walk on a simple symmetric exclusion process, is one of the most related ones to ours. Scaling limits have been derived in \cite{AJV,ASV,HKT,HS}. In this model, a tagged particle performs a (continuous-time or discrete-time) simple random walk with rates depending on the occupancy of red particles at its location, while red particles follow SEP. It is similar to our model except that red particles are not affected by the tagged particle, jump rates for the random walk depend on their location instead of their targets, and there is no removal of red particles. Regenerative structures are observed in \cite{ASV,HS} when the rates of random walk satisfy conditions on drift and ellipticity. In general, conditions on drift and ellipticity are to ensure the tagged particle moves fast regardless presence of red particles. As most red particles move diffusively, we expect regenerative times. In \cite{HKT}, it is shown that the drift condition can be removed for the discrete-time random walk case for law of large numbers, and a central limit theorem will hold for $X_t$ if its speed is nonzero. Till now, the problem remains open when jump rates are non-elliptic, or when jumps are completely suppressed on red particles.}
		
		\textcolor{black}{Large deviation principles have been obtained for some fairly general RWDRE. 
		 For random walks on interacting particle system, there are also results, such as \cite{A,AJV,BR1,HS}.  In some scenarios, we characterize the rate function with subadditive arguments, couplings and convex analysis. A drawback is that the formula is often implicit. In fact, we seldom have explicit formula for non-solvable models. A major issue is that the environment process usually depends on their parameters in a nontrivial way. In \cite{AJV} Avena, Jara and Vollering obtained an explicit LDP for $X_t$ when the random walk and SEP are jointly scaled appropriately. Their approach is via hydrodynamic limits and a perturbation argument. They can perturb the joint dynamics of the environment and the tagged particle so that the cost can be expressed in terms of the hydrodynamic limit of the perturbed system. Our work is in spirit similar to theirs, and we also get the rate function explicitly. A major difference is that in our case, the marginal distribution of ergodic measure is explicit and it depends on parameters around site 1. We can perturb the dynamics locally near the tagged particle with explicit cost in terms of local statistics, which greatly helps the analysis. Also, our tagged particle affects the environment.}


\section{Notations and Main Results} 
Since red particles are removed when they are on the left of the tagged particle, we consider only the red particles to the right of the tagged particle. A configuration $\xi(\cdot)$ on $\mathbb{Z}_+ = \mathbb{N}\setminus\{0\}$ indicates which sites are occupied relative to the tagged particle:  $\xi(x) = 1$ if site $x$ is occupied, and $\xi(x) = 0$ otherwise. The collection of all configurations $\mathbb{X}=\{0,1\}^{\mathbb{Z}_+}$ forms a state space for the process $\xi_t$. 

Local functions on ${\mathbb{Z}_+}$ are functions defined on $\mathbb{X}$ and they depend on finitely many $\xi(x)$. Examples of local functions are
$\xi_x$ and $\xi_A$: 
\begin{align}
\xi_x(\xi) =& \xi(x)\\
\xi_A(\xi) =& \prod_{x\in A} \xi(x), \text{ A is a finite set of } \mathbb{Z}_+
\end{align}
We denote by $\mathbi{C}$ the space of local functions on ${\mathbb{Z}_+}$, and denote by $\mathbf{M}_1$ the space of probability measures on $\mathbb{X}$.

The environment process $\xi_t$ starting from any initial configuration in $\mathbb{X}$ is a well-defined Markov processes. It is described by generator $\mathit{L}_d=\mathit{S}^{ex}_+ +\mathit{L}^{sh} + \mathit{L}^d$ on local functions, and the action of $\mathit{L}_d$ on any local function $f$ is given by:
\begin{align}
\mathit{L}_df(\xi) =& (\mathit{S}^{ex}_+ +\mathit{L}^{sh} + \mathit{L}^d) f(\xi) \notag \\
=&\sum_{x,y> 0}p(y-x)\xi_x\left( 1-\xi_y\right)\left(f(\xi^{x,y})-f(\xi)\right) \notag \\
&+\sum_{z}q(z)\left( 1-\xi_z\right)\left(f(\theta_z\xi)-f(\xi)\right) \notag \\
&\sum_{x > 0 >y }p(y-x)\xi_x \left(f(\xi^{x})-f(\xi)\right) 
\label{Eq:generator L_d}
\end{align}
where $\xi^{x,y}$ represents the configuration after exchanging particles at site $x$ and $y$ of $\xi$,	
\begin{equation} \label{usual exchange}
\xi^{x,y}(z) =
\begin{cases}
\xi(z) & \text{if }z \neq x,y \\
\xi(y) & \text{if }z = x \\
\xi(x) & \text{if }z = y. 
\end{cases}
\end{equation}
 $\theta_z\xi$ represents the configuration shifted by $-z$ unit due to the jump of the tagged particle to an empty site at $z$,
\begin{equation} \label{usual translation}
(\theta_z\xi)(x) =
\begin{cases}
\xi(x+z) & \text{if }x \neq -z \\
\xi(z) & \text{if }x = -z. 
\end{cases}
\end{equation} 
and  $\xi^{x}$ represents the configuration after changing the value at  site $x$,
\begin{equation} \label{change sign at site x}
\xi^{x}(z) =
\begin{cases}
\xi(z) & \text{if }z \neq x \\
1-\xi(z) & \text{if }z = x. 
\end{cases}
\end{equation}

We denote by  $\mathbb{P}^{\eta,d}$ the probability measure on the space of c\'{a}dl\'{a}g paths on $\mathbb{X}$ when the initial configuration is deterministic $\xi_0=\eta$.
  Let $\mathbb{P}^{\nu_0,d}=\int\mathbb{P}^{\eta,d}\, d\nu_0(\eta)$  when $\xi_0$ is distributed according to some measure $\mu$ on $\mathbb{X}$. We denote by $\mathbb{E}^{\nu_0,d}$ the expectation with respect to $\mathbb{P}^{\nu_0,d}$.

For the purpose of this paper, we will consider the case when the red particles can jump two steps, tagged particle can only jump to the right with one step, and the initial measures are Bernoulli product measures with parameters $\rho$ for the process. That is, $p(\cdot),q(\cdot)$ and $\nu_0$ satisfy
\begin{enumerate}[label=A\arabic*]
	\item (Range Two, Symmetric) $p(2)=p_2>0$, $ p(x)=0$ for $x>2$, and $p(x)= p(-x)$.
	\item (Right Nearest-neighbor Jump) $q(1)=q_1>0$, and $ q(x)=0$ otherwise.
	\item (Bernoulli Initial Measure) $\nu_0 = \mu_\rho$, which  is a product measure on $\mathbb{X}=\{0,1\}^{\mathbb{Z}_+}$, with marginals $<\mu_\rho,\eta_x> = \rho$ for all $x >0$. 
\end{enumerate}
 
The first theorem says that the tagged particle in the SEP with a removal rule has a speed determined by $\rho, p_2,q_1$ ($p_1$ is not involved):
 \begin{theorem}\label{Thm: LLN for X_t}
 	Consider a driven tagged particle in the SEP with removal rules. Assume jump rates $p(\cdot),q(\cdot)$ and initial measure $\nu_0$ satisfy assumptions A1,A2 and A3. Then the displacement $X_t$ of the tagged particle satisfies a law of large numbers with a speed $m =\frac{p_2q_1}{p_2 + \rho q_1 } $,
 	
 	\begin{equation}\label{Eq: LLN for X_t}
 	\lim_{t\to\infty}\frac{X_t}{t} = m=\left(\frac{1}{q_1} + \frac{\rho}{p_2}\right)^{-1}, \quad \mathbb{P}^{\mu_\rho,d} - a.s.
 	\end{equation} 
 \end{theorem}
 
 \begin{remark}\begin{enumerate}
 		\item 	We can extend this result to the case where the symmetric jump rate $p(\cdot)$ has any finite support and the case where red particles are removed if they jump to the left of the tagged particle with a distance $D \geq 1$. There is a unique (implicit) speed for the tagged particle. The main assumptions are that $D<\infty$ and  $q(\cdot)$ has only right jumps. 
 		\item The speed formula \eqref{Eq: LLN for X_t} does not involve the parameter $p_1$. In fact, the large deviation rate function does not involve $p_1$ either, see \eqref{eq: rate function for joint} and \eqref{eq: rf from contraction principle}. This is special when $p(\cdot)$, $q(\cdot)$ satisfy assumptions A1, A2.
 	\end{enumerate}
 \end{remark}
%

 The second theorem says that if the tagged particle has a large enough speed, the displacement $X_t$ satisfies a functional central limit theorem:
  \begin{theorem}\label{Thm: CLT for X_t}
  	Under the assumption of Theorem \ref{Thm: LLN for X_t}, and if further, the speed $m$ is strictly larger than the drift $w$,
  	\[m= \left(\frac{1}{q_1}+\frac{\rho}{p_2}\right)^{-1} > p_1+3p_2 =w,\]
  	then there is a $\sigma >0$, such that under $\mathbb{P}^{\mu_\rho,d}$,
  	
  	\begin{equation}
		\left(\frac{X_{nt}- mnt}{\sqrt{n}}\right)_{t\geq 0} \Longrightarrow \sigma B_t
  	\end{equation} 
  	where $B_t$ is a standard Brownian motion.
  \end{theorem}

The third theorem says that there is a good rate function $\mathbb{I}(\cdot)$, such that, the displacement $X_t$ satisfies the large deviation principle (LDP):
\begin{theorem}\label{thm: LDP}
Under the assumption of Theorem \ref{Thm: LLN for X_t}, the dislacement $X_t$ has the LDP with a good rate function $\mathbb{I}(\cdot)$. That is, for any Borel set $C \subset \mathbb{R}$,
\begin{align} -\inf_{y\in C^o} \mathbb{I}(y) &\leq \liminf_{y\to \infty}  \frac{1}{t} \ln \mathbb{P}^{{\mu}_\rho,d}\left(\frac{X_t}{t} \in C \right) \notag \\
&\leq \limsup_{t\to \infty}  \frac{1}{t} \ln \mathbb{P}^{{\mu}_\rho,d}\left(\frac{X}{t} \in C \right) \leq -\inf_{y\in \overline{C}} \mathbb{I}(y) .
\end{align}
Particularly, the rate function $\mathbb{I}(\cdot)$ can be computed explicitly. 
\end{theorem}

We briefly discuss the approaches to the results and the organization of the paper.

We first use the graphical construction of the symmetric exclusion process and construct auxiliary processes by using two color schemes. Due to symmetric jump rates $p(\cdot)$, we can view the $S^{ex}_+$ as interchanging information between sites, and the process of interchanging information is almost independent of the initial configuration $\xi_0$.  With a Bernoulli initial measure $\mu_\rho$, we can view every site start with a cup carrying a "Bernoulli" particle initially. Each "Bernoulli" particle is revealed due to $L^{sh}$, $L^d$, and the cup is colored according to revealing. By analyzing the colors of cups, we get different estimates. The auxiliary processes and color schemes will be introduced in section \ref{sec: graphical construction and coloring}.

The first color scheme allows us to get estimates for Theorem \ref{Thm: LLN for X_t} and Theorem \ref{thm: LDP}. In this auxiliary process, revealing and coloring happen only at attempts of jumps of the tagged particle and jumps of red particles towards the negative axis. \textcolor{black}{These jumps occur near site 1. We will see that there is an invariant measure, under which the total number of colored cups on positive axis with site 1 vacant is finite. This estimate enables us to show that this invariant measure is ergodic and to compute its marginal distribution at site 1 explicitly.} The related estimates and proof of Theorem \ref{Thm: LLN for X_t} will be done in section \ref{sec: LLN for displacement}. On the other hand, the large deviation principle can be viewed as a consequence of the law of large numbers and perturbation arguments. We will consider various jumps and empirical densities at site 1, and we will show their joint LDP. By the contraction principle, we can obtain the LDP for the displacement. This will be done in section \ref{sec: LDP}.

The second color scheme allows us to define a regeneration time and estimate its moments. We \textcolor{black}{will} define a boundary $m_t$, which increases with steps of sizes $1$ and $2$ at rates $p_1+p_2$ and  $p_2$, and decreases when the tagged particle jumps. Particles and cups on $(0,m_t]$ are revealed and colored, respectively, while particles and cups on $(m_t,\infty)$ remain unrevealed "Bernoulli" particles and white, respectively. The regeneration time $\tau$ is the first time when $m_t =0$. With the help of exponential martingales, we can compute the moment generating functions of $\tau$ and $X_\tau$, from which the functional central limit theorem for $X_t$ follows. This will be done in section \ref{sec: Regenerative Structure}.  

\section{Auxiliary Processes and Color Schemes}\label{sec: graphical construction and coloring}
\textcolor{black}{It is known that the displacement $X_t$ is the same as the number of shifts of the environment process $\xi_t$. In this section, we construct three auxiliary processes as variants of the environment process $\xi_t$ with the graphical construction. The second and third auxiliary processes are extensions of the first one and they corresponds to two color schemes. With these two auxiliary processes, we can get some martingales which allows us to obtain estimates for limit theorems of $X_t$. For computations in sections \ref{sec: LLN for displacement}, \ref{sec: Regenerative Structure}, and \ref{sec: LDP}, we will also mention the generators of the auxiliary processes.}

Due to symmetric jump rates $p(\cdot)$, we can rewrite the generator $S^{ex}_+$ as
\begin{align}
	 \mathit{S}^{ex}_+ f(\xi)
	 =&\sum_{x,y> 0}p(y-x)\xi_x\left( 1-\xi_y\right)\left(f(\xi^{x,y})-f(\xi)\right) \notag\\
	 =&\sum_{x>y> 0}p(y-x)\left(\xi_x\left( 1-\xi_y\right)+\xi_y\left( 1-\xi_x\right) \right)\left(f(\xi^{x,y})-f(\xi)\right) \notag \\
	 =&\sum_{x>y> 0}p(y-x)\left(f(\xi^{x,y})-f(\xi)\right) \label{eq: Symmetric part} 
\end{align} This also corresponds to the interchange (stirring) process. See \cite{ASV} or Chapter VIII.4 \cite{Li85}  for interchange process. \textcolor{black}{We can apply a similar idea to get an enriched system $\zeta_t = (c_t,l_t,\xi_t)$ of the original environment process $\xi_t$. Particularly, the $\xi_t$ in both systems are the same.}

\textcolor{black}{We construct a basic auxiliary process $\zeta_t = (c_t,l_t,\xi_t)$ with graphical construction} as follows. Consider a collection of cups labeled by their initial positions on $\mathbb{Z}_+$. \textcolor{black}{Initially, every cup is colored white, and it contains either a red or yellow particle, which represents an occupied site (1) or a vacant site (0). The colors of cups and particles inside will be changed at certain event times.} Let  $(\mathcal{N}_{x,y})_{x> y>0}, \mathcal{C}$ and $\mathcal{D}$ be a collection of independent Poisson processes with rates $(p(x,y))_{x>y>0}, q_1$ and $p_2$. 
At an event time $t$ of $\mathcal{N}_{x,y}$, we interchange the cups at sites $x$ and $y$ together with the particles they contain. An event time $t$ of $\mathcal{C}$ is an attempt of a jump of the tagged particle.  If there is a yellow particle at site 1, the jump is successful, and we remove everything at site 1 and shift the configuration to the left by 1; otherwise, we color the cup at site 1 by blue (b). At an event time $t$ of $\mathcal{D}$, we always replaced the particle at site 1 by a yellow particle, and color the cup purple (p). We denote by $c_t = (c_t(i))_{i>0}$ the colors of cups at each site, and by $l_t = (l_t(i))_{i>0}$ the labels of cups at each site. We denote by $\mathbb{Q}^{\nu_0,d}$ the corresponding probability measure for this auxiliary process when $\nu_0$ is the distribution of $\xi_0$.
See Figure 1 for an example. In this example, $\zeta_{t-}(1) = (w,5,1)$, $\zeta_{t-}(2) = (w,10,0)$, and $\zeta_{t-}(3) = (b,7,1)$.

\begin{figure}[h!]
	\begin{center}
\begin{tikzpicture} 
\draw[black, thick] (-4.3,0) -- (4.3,0) node[black,right] {$\zeta_{t-}$}; 
\draw [black,fill] (-4-0.1,0.1) rectangle (-4+0.1,-0.1) node [black,below=4] {Tagged}; 

\foreach \x in {-3,...,4}
\draw (\x,0) circle (0.1);

\foreach \x in {-3,2}
\draw [black,fill](\x,0) circle (0.1)
node [black,above=4]{w};

\foreach \x in {-1,3}
\draw [black,fill](\x,0) circle (0.1)
node [black,above=4]{b};

\foreach \x in {0,4}
\draw (\x,0) circle (0) node [black,above=4]{p};

\foreach \x in {-2,1}
\draw (\x,0) circle (0) node [black,above=4]{w};

\draw (-3,0) circle (0) node [black,below=4]{5};
\draw (-2,0) circle (0) node [black,below=4]{10};
\draw (-1,0) circle (0) node [black,below=4]{7};
\draw (0,0) circle (0) node [black,below=4]{8};
\draw (1,0) circle (0) node [black,below=4]{13};
\draw (2,0) circle (0) node [black,below=4]{2};
\draw (3,0) circle (0) node [black,below=4]{1};
\draw (4,0) circle (0) node [black,below=4]{6};

\end{tikzpicture}

\begin{tikzpicture}
\draw[black, thick] (-4.3,0) -- (4.3,0) node[black,right] {$\zeta_{t}$}; 
\draw [black,fill] (-4-0.1,0.1) rectangle (-4+0.1,-0.1) node [black,below=4] {Tagged}; 
\foreach \x in {-3,...,4}
\draw (\x,0) circle (0.1);

\foreach \x in {-3,2}
\draw [black,fill](\x,0) circle (0.1)
node [black,above=4]{w};

\foreach \x in {1,3}
\draw [black,fill](\x,0) circle (0.1)
node [black,above=4]{b};

\foreach \x in {0,4}
\draw (\x,0) circle (0) node [black,above=4]{p};

\foreach \x in {-2,-1}
\draw (\x,0) circle (0) node [black,above=4]{w};

\draw (-3,0) circle (0) node [black,below=4]{5};
\draw (-2,0) circle (0) node [black,below=4]{10};
\draw (-1,0) circle (0) node [black,below=4]{13};
\draw (0,0) circle (0) node [black,below=4]{8};
\draw (1,0) circle (0) node [black,below=4]{7};
\draw (2,0) circle (0) node [black,below=4]{2};
\draw (3,0) circle (0) node [black,below=4]{1};
\draw (4,0) circle (0) node [black,below=4]{6};
\draw[dashed,<->] (-1,0.7) to (-1,1.1) to (1,1.1)
to (1,0.7);
\end{tikzpicture}

\begin{tikzpicture}
\draw[black, thick] (-4.3,0) -- (4.3,0) node[black,right] {$\zeta_{t}$}; 
\draw [black,fill] (-4-0.1,0.1) rectangle (-4+0.1,-0.1) node [black,below=4] {Tagged}; 

\foreach \x in {-3,...,4}
\draw (\x,0) circle (0.1);

\foreach \x in {2}
\draw [black,fill](\x,0) circle (0.1)
node [black,above=4]{w};

\foreach \x in {-3,-1,3}
\draw [black,fill](\x,0) circle (0.1)
node [black,above=4]{b};

\foreach \x in {0,4}
\draw (\x,0) circle (0) node [black,above=4]{p};

\foreach \x in {-2,1}
\draw (\x,0) circle (0) node [black,above=4]{w};

\draw (-3,0) circle (0) node [black,below=4]{5};
\draw (-2,0) circle (0) node [black,below=4]{10};
\draw (-1,0) circle (0) node [black,below=4]{7};
\draw (0,0) circle (0) node [black,below=4]{8};
\draw (1,0) circle (0) node [black,below=4]{13};
\draw (2,0) circle (0) node [black,below=4]{2};
\draw (3,0) circle (0) node [black,below=4]{1};
\draw (4,0) circle (0) node [black,below=4]{6};
\draw[dashed,->]  (-3.5,1.1) to (-3.2,0.4);
\end{tikzpicture}

\begin{tikzpicture}
\draw[black, thick] (-4.3,0) -- (4.3,0) node[black,right] {$\zeta_{t}$}; 
\draw [black,fill] (-4-0.1,0.1) rectangle (-4+0.1,-0.1) node [black,below=4] {Tagged}; 

\foreach \x in {-3,...,4}
\draw (\x,0) circle (0.1);

\foreach \x in {2}
\draw [black,fill](\x,0) circle (0.1)
node [black,above=4]{w};

\foreach \x in {-1,3}
\draw [black,fill](\x,0) circle (0.1)
node [black,above=4]{b};

\foreach \x in {-3,0,4}
\draw (\x,0) circle (0) node [black,above=4]{p};

\foreach \x in {-2,1}
\draw (\x,0) circle (0) node [black,above=4]{w};

\draw (-3,0) circle (0) node [black,below=4]{5};
\draw (-2,0) circle (0) node [black,below=4]{10};
\draw (-1,0) circle (0) node [black,below=4]{7};
\draw (0,0) circle (0) node [black,below=4]{8};
\draw (1,0) circle (0) node [black,below=4]{13};
\draw (2,0) circle (0) node [black,below=4]{2};
\draw (3,0) circle (0) node [black,below=4]{1};
\draw (4,0) circle (0) node [black,below=4]{6};
\draw[dashed,->]  (-3,0.7) to (-3,1.0) to (-4.2,1.0);
\end{tikzpicture}
\caption{Configurations $\zeta$ before and after Event Times of $\mathcal{N}_{3,5}(t)$,$\mathcal{C}(t)$, and $\mathcal{D}(t)$}
\end{center}
\end{figure}
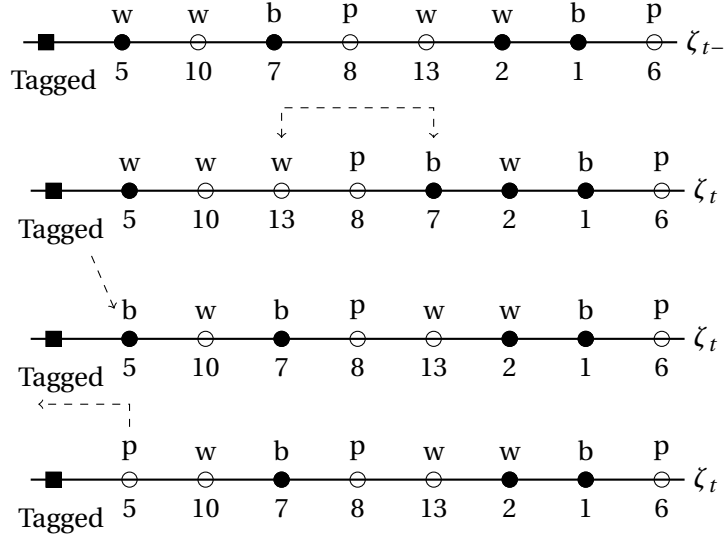
For this auxiliary process, \textcolor{black}{ the coloring of cups only occurs at site $1$ and at event times of $\mathcal{C}$ and $\mathcal{D}$. At any time $t$, a white cup with a label $j$ contains the same particle as it initially does.} Since the initial measure is a Bernoulli product measure, we can view white cups corresponding to independent "Bernoulli" particles:
\begin{lemma} \label{lm: Bernoulli particles}
	Consider the auxiliary process $\zeta_t$ with the initial configuration $\eta_0 = (c_0, l_0,\xi_0)$ such that 
	${c_0(i) = w}$, ${l_0(i)=i}$ for all $i$, and $\xi_0$ is distributed according to the Bernoulli measure $\mu_\rho$. For any finite set $A \subset \mathbb{Z}_+$, and any $t\geq 0$,
	\begin{equation}\label{eq:indep}
		\mathbb{Q}^{\mu_\rho,d}(\xi_A(t)=1 \arrowvert c_t(i)= w, \text{ for all } i \text{ in }A ) = \rho ^{\abs{A}}.
	\end{equation} 
\end{lemma}
\begin{proof}:
	\textcolor{black}{By the graphical construction}, at any time $t\geq 0$, for any $i,j$ 
	\begin{equation*}
		\xi_t(i) =1, l_t(i) = j, c_t(i)= w \Longleftrightarrow \xi_0(j) =1, l_0(j) =j, l_t(i) = j, c_t(i) = w 
	\end{equation*}
	A white cup at site 1 is always removed or colored at event times $s$ of $\mathcal{C},\mathcal{D}$. Therefore, any white cup remained at time $t$ is not at site 1 at any event time $s$ of $\mathcal{C}$ or $\mathcal{D}$ for $s\leq t$. As the particle in a white cup depends on its initial state \textcolor{black}{and Poisson processes are independent of $\xi_0$}, we have 
	\begin{align}
		&\mathbb{Q}^{\mu_\rho,d}\left(\xi_0(j) =1, l_0(j) =j, c_t(i) = w, \text{ for all } i \in A\right)\notag \\
		=&\mathbb{Q}^{\mu_\rho,d}\left(\xi_0(j) =1, l_0(j) =j,l_t(i) = j, c_t(i) = w, \text{ for all } i \in A\right)\notag \\
		=& \mu_\rho(\xi_0(j) =1, \text{ for all } i \in A )\cdot \notag \mathbb{Q}^{\mu_\rho,d}\left(l_0(j) =j,l_t(i) = j, c_t(i) = w, \text{ for all } i \in A\right) \notag \\
		=& \rho^{\abs{A}} \cdot \mathbb{Q}^{\mu_\rho,d}\left(l_0(j) =j,l_t(i) = j, c_t(i) = w, \text{ for all } i \in A\right)
	\end{align}
	Summing over j, we get \eqref{eq:indep}.
\end{proof}

\textcolor{black}{With the above lemma, we can get the distribution of $\xi_t$ from the distribution of colors of cups $c_t$ at time $t$, see Remark \ref{rmk: c_t,xi_t} at the end of this section. This enables us to ignore the labels and the particles inside cups and to consider only dynamics of the colors of cups.} We construct two further auxiliary processes with two different color schemes of cups. \textcolor{black}{In these auxiliary processes, non-white cups remain non-white.}

	In the first auxiliary process, we denote by $\eta_t$ = $(c_t(i),s_t(i))_{i>0}$ the colors of cups and types of particles. There are three colors of cups, white (w), blue (b), and purple (p). The color process $c_t$ is the same as that in the basic auxiliary process (by the graphical construction to be introduced). There are also three types of particles, "Bernoulli" particles (B), red particles ($1$), and yellow particles ($0$). The red particles and yellow particles are also called "revealed" Bernoulli particles. \textcolor{black}{We will describe how to color cups and reveal particles so that we can identify $\eta_t$ with $c_t$ (or $s_t$). 
		That is, if an initial configuration $\eta_0 = (c_0,s_0)$ satisfies for any $i \in \mathbb{Z}_+$, \begin{equation} \label{eq: color and states}
		"c_0(i)= p \Leftrightarrow s_0(i) = 0 "\text{ and } "c_0(i)= b \Leftrightarrow s_0(i) = 1",  
		\end{equation} the configuration $\eta_t$ at any latter time $t>0$ also satisfies \eqref{eq: color and states}. Particularly, this condition works for the initial configuration \begin{equation}\label{eq: initial condition for Ld,1}
		\eta_0(i)= (w,B) \text{ for all } i >0.
		\end{equation}
		Actually, this is direct if we do the following.} We have the same independent Poisson processes $(\mathcal{N}_{x,y})_{x> y>0}$, $\mathcal{C}$ and $\mathcal{D}$ as those in the basic auxiliary processes. At an event time $t$ of $\mathcal{N}_{x,y}$, $\eta_t(x)$ interchanges with $\eta_t(y)$. At an event time $t$ of $\mathcal{C}$, if $s_{t-}(1) = 0$, we shift $\eta_t$ to the left by 1: $\eta_t = \theta_1 \eta_{t-}$; if $s_{t-}(1) = B$, with a probability $\rho$, the Bernoulli particle is revealed as a red particle, and the white cup is colored blue: $\eta_t = C_{b,1} \eta_{t-}$, and with a probability $1-\rho$, the Bernoulli particle is revealed as a yellow particle, and we shift $\eta_t$ to left by 1: $\eta_t = \theta_1\circ C_{p,1} \eta_{t-}$; if $s_{t-}(1) =1$, we do nothing. At an event time $t$ of $\mathcal{D}$, the particle at site 1 is replaced by a yellow particle, and the cup is colored purple: $\eta_t = C_{p,1} \eta_{t-}$. The operators $C_{p,j}$ and $C_{b,j}$ are defined by:
		\begin{align}
		C_{p,j} \eta(i) =\begin{cases}
		(p,0) &,i=j, \\
		\eta(i) &, i\neq j .
		\end{cases} \label{eq: C_p}\\
		C_{b,j} \eta(i) =\begin{cases}
		(b,1) &,i=j, \\
		\eta(i) &, i\neq j .
		\end{cases} \label{eq: C_b,j}
		\end{align}
		See Figure 2 for an example. The dashed boxes represent concealed "Bernoulli" particles. In this example, $\eta_t = \eta^{3,5}_{t-}$, $C_{b,1}\eta_{t-}$, and $C_{p,1}\eta_{t-}$ respectively.
		
	\begin{figure}[h!]
		\begin{center}
			\begin{tikzpicture} 
			\draw[black, thick] (-4.3,0) -- (4.3,0) node[black,right] {$\eta_{t-}$}; 
			\draw [black,fill] (-4-0.1,0.1) rectangle (-4+0.1,-0.1) node [black,below=4] {Tagged}; 
			
			\foreach \x in {-2,1,-3,2}
			\draw [dashed,black] (\x-0.2,0.2) rectangle (\x+0.2,-0.2)
			 [dashed,black] (\x,0) circle (0)
			node [black,above=5]{w};
			
			\foreach \x in {-1,3}
			\draw [black,fill](\x,0) circle (0.1)
			node [black,above=4]{b};
			
			\foreach \x in {0,4}
			\draw (\x,0) circle (0.1) node [black,above=4]{p};

			\end{tikzpicture}
			
			\begin{tikzpicture}
			\draw[black, thick] (-4.3,0) -- (4.3,0) node[black,right] {$\eta_{t}$}; 
			\draw [black,fill] (-4-0.1,0.1) rectangle (-4+0.1,-0.1) node [black,below=4] {Tagged}; 
			
			\foreach \x in {-2,-1,-3,2}
			\draw [dashed,black] (\x-0.2,0.2) rectangle (\x+0.2,-0.2)
			[dashed,black] (\x,0) circle (0)
			node [black,above=5]{w};
			
			\foreach \x in {1,3}
			\draw [black,fill](\x,0) circle (0.1)
			node [black,above=4]{b};
			
			\foreach \x in {0,4}
			\draw (\x,0) circle (0.1) node [black,above=4]{p};
			\draw[dashed,<->] (-1,0.7) to (-1,1.1) to (1,1.1)
			to (1,0.7);
			\end{tikzpicture}
			
			\begin{tikzpicture}
			\draw[black, thick] (-4.3,0) -- (4.3,0) node[black,right] {$\eta_{t}$}; 
			\draw [black,fill] (-4-0.1,0.1) rectangle (-4+0.1,-0.1) node [black,below=4] {Tagged}; 
			
			\foreach \x in {-2,1,2}
			\draw [dashed,black] (\x-0.2,0.2) rectangle (\x+0.2,-0.2)
			[dashed,black] (\x,0) circle (0)
			node [black,above=5]{w};
			
			\foreach \x in {-1,3,-3}
			\draw [black,fill](\x,0) circle (0.1)
			node [black,above=4]{b};
			
			\foreach \x in {0,4}
			\draw (\x,0) circle (0.1) node [black,above=4]{p};
			\draw[dashed,->]  (-3.5,1.1) to (-3.2,0.4);
			\end{tikzpicture}
			
			\begin{tikzpicture}
			\draw[black, thick] (-4.3,0) -- (4.3,0) node[black,right] {$\eta_{t}$}; 
			\draw [black,fill] (-4-0.1,0.1) rectangle (-4+0.1,-0.1) node [black,below=4] {Tagged}; 
			
			\foreach \x in {-2,1,2}
			\draw [dashed,black] (\x-0.2,0.2) rectangle (\x+0.2,-0.2)
			[dashed,black] (\x,0) circle (0)
			node [black,above=5]{w};
			
			\foreach \x in {-1,3}
			\draw [black,fill](\x,0) circle (0.1)
			node [black,above=4]{b};
			
			\foreach \x in {0,4,-3}
			\draw (\x,0) circle (0.1) node [black,above=4]{p};
			\draw[dashed,->]  (-3,0.7) to (-3,1.1) to (-4.2,1.1);
			\end{tikzpicture}
			\caption{Configurations $\eta$ before and after Event Times of $\mathcal{N}_{3,5}(t)$,$\mathcal{C}(t)$ with Revealing as a Particle, and $\mathcal{D}(t).$}
		\end{center}
	\end{figure}
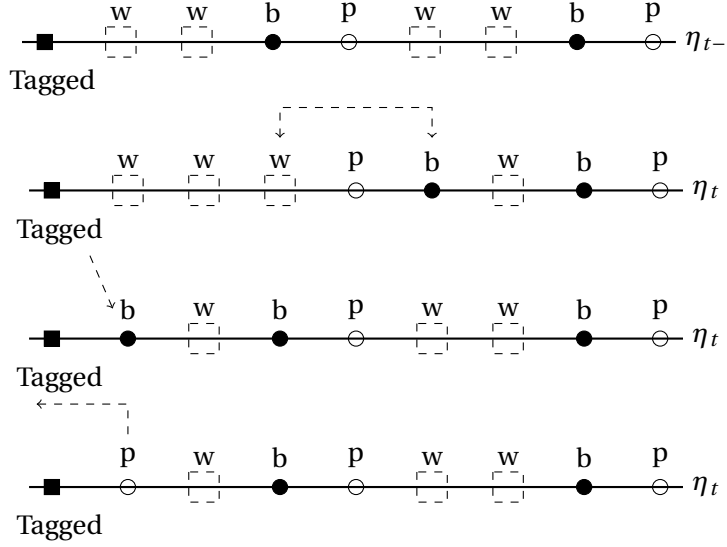
 \textcolor{black}{It is not hard to see from the graphical construction, $c_t$ in $\eta_t$ is the same process as $c_t$ in $\zeta_t$  when the initial configuration satisfies \eqref{eq: initial condition for Ld,1}.} We can also write the generator $\tilde{L}_{d,1}$ for this auxiliary process \textcolor{black}{from events at $(\mathcal{N}_{x,y})_{x> y>0}$, $\mathcal{C}$, and $\mathcal{D}$.} $\tilde{L}_{d,1} =\tilde{S}^{ex}_{+,1} + \tilde{L}^{sh} + \tilde{L}^{d}$ acts on a local function $f$ as
	\begin{equation} \label{eq: L_d,1}
		\tilde{L}_{d,1}f(\eta) =\left(\tilde{S}^{ex}_{+,1} + \tilde{L}^{sh} + \tilde{L}^{d}\right)f(\eta) 
		\end{equation} 
	\begin{equation}
		\label{eq: new S_^ex,1}
		\tilde{S}^{ex}_{+,1}f(\eta)=  \sum_{x>y> 0}p(y-x)\left(f(\eta^{x,y})-f(\eta)\right) 
		\end{equation}
	\begin{equation}		\label{eq: new L_^d}
	\tilde{L}^{d}f(\eta) = p_2 \cdot\mathbb{1}_{\{c(1)\neq p\}} \left(f(C_{p,1}\eta)-f(\eta)\right)
	\end{equation}
		\begin{align} 
		\label{eq: new L_^sh}
		 \tilde{L}^{sh} f(\eta)=& (1-\rho)\cdot q_1\cdot \mathbb{1}_{\{c(1)=w\}}\left(f(\theta_1\circ C_{p,1}\eta)-f(\eta)\right) \notag \\
		&+ \rho \cdot q_1\cdot \mathbb{1}_{\{c(1)=w\}}\left(f(C_{b,1}\eta)-f(\eta)\right) \notag \\
		&+ q_1\cdot \mathbb{1}_{\{c(1)=p\}}\left(f(\theta_1\eta)-f(\eta)\right) .	\end{align}
		\textcolor{black}{At this point, we should notice that $\rho$ is in the generator of $\tilde{L}_{d,1}$ but not in $\mathit{L}_d$. This plays a role in the proof of the LDP of $X_t$.}	
 \textcolor{black}{Meanwhile, the event times of $\mathcal{C}$ are exactly the times when the tagged particle attempts to jump, and a successful jump only occurs when the cup at site 1 is purple or when the cup is white and the "Bernoulli" particle is revealed as a yellow particle. In either case, a left shift of the configuration follows. Therefore, we can recover $X_t$ by tracing shifts of $c_t$ at event times of $\mathcal{C}$.}

The second auxiliary process $(\psi_t,m_t) = ((\tilde{c}_t(i), \tilde{s}_t(i))_{i>0},m_t)$ is similar to $\eta_t$. \textcolor{black}{We also want identify $\tilde{c}_t$ as $\tilde{s}_t$ by ensuring \eqref{eq: initial condition for Ld,1} holds for all time.} \textcolor{black}{We use $\tilde{c}_t, \tilde{s}_t$ to emphasize that the colors of cups and types of particles are different from $c_t$ and $s_t$. These differences are due to the additional revealing of "Bernoulli" particles with an artificial boundary process $m_t$.} \textcolor{black}{To summarize, the dynamics are the following. We reveal new "Bernoulli" particles and color white cups by increasing $m_t$. We remove revealed particles and purple cups at event times of $\mathcal{C}$. We change colors of cups at event times of $\mathcal{D}$.} 

\textcolor{black}{More precisely, we have the same independent Poisson processes $(\mathcal{N}_{x,y})_{x> y>0}$, $\mathcal{C}$, $\mathcal{D}$ and also an additional independent Poisson process $\mathcal{N}_{0,2}$ of rate $p_2$. This additional $\mathcal{N}_{0,2}$ is artificial, and it ensures $m_t$ increases in time homogeneously.} Let $m_t$ be the rightmost site with a non-white cup in $\psi_t$: ${m_t := \sup\{i:\tilde{c}(i)\neq w\}\vee 0}$, and every particle on sites $(0,m_t]$ is revealed. When $m_{t-}>0$ and at an event time $t$ of $\mathcal{N}_{x,y}$ with $0<x\leq m_{t-} <y$, we increase $m_t$ to $y$, reveal all "Bernoulli" particles on $(m_{t-},y]$ according to i.i.d Bernoulli trials, and color the cups accordingly. We then interchange cups and particles on sites $x$ and $y$. \textcolor{black}{Suppose the colors of cups and the types of particles on sites $(x,y]$ after revealing are $\hat{c}_{t-}$ and $\hat{s}_{t-}$ (which are different from $\tilde{c}_{t-}$ and $\tilde{s}_{t-}$ because of revealing), the new configuration at time $t$ is ${(\psi_t,m_t) = (\hat{c}^{x,y}_{t-},\hat{s}^{x,y}_{t-},y)}$.} When ${m_{t-}=1}$ and at an event time $t$ of $\mathcal{N}_{0,2}$, we increase $m_t$ to 2, reveal particles, and color cups accordingly, but we do not interchange particles or cups: $(\psi_t,m_t) = (\hat{c}_{t-},\hat{s}_{t-},2)$. At event times of $\mathcal{C}$ and $\mathcal{D}$, we use the same color scheme as the first auxiliary process. Lastly, we only decrease $m_t$ by 1 at an event time $t$ of $\mathcal{C}$ when a yellow particle is at site 1. For convenience, we initially reveal the particle at site 1 and set the rest sites with white cups containing independent "Bernoulli" particles:
\begin{equation}\label{eq: renewal process initial condition}
	\tilde{c}_0(i) = w, \tilde{s}_t(i)=B \text{ for all } i> 1, m_0=1
\end{equation} See Figures 3 and 4 for examples. In Figure 3, $m_{t-}=4$, $\psi_t = \left(C_{p,5}\psi_{t-}\right)^{3,5}$, $\theta_1 \psi_{t-}$ and $\psi_{t-}$ respectively. In Figure 4, $m_{t-}=1$, $\psi_t = \left(C_{b,2}\psi_{t-}\right)^{1,2}$, and $C_{b,2} \psi_{t-}$ respectively.
	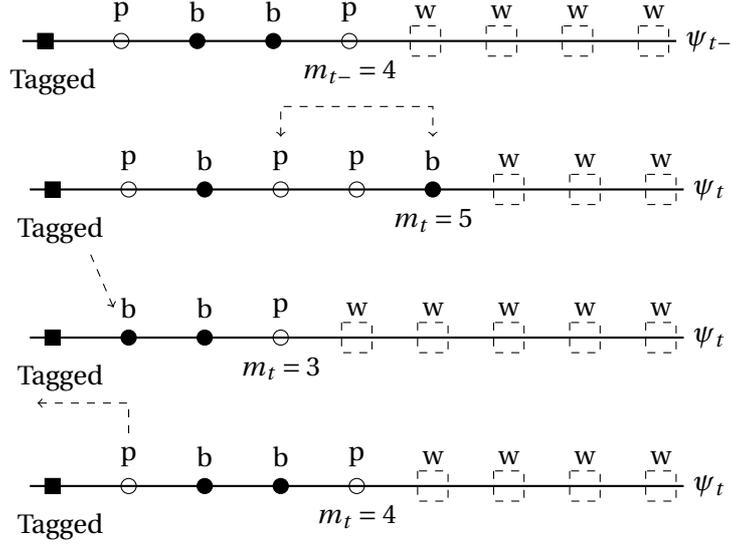
\begin{figure}[h!]
		\begin{center}
			\begin{tikzpicture} 
			\draw[black, thick] (-4.3,0) -- (4.3,0) node[black,right] {$\psi_{t-}$}; 
			\draw [black,fill] (-4-0.1,0.1) rectangle (-4+0.1,-0.1) node [black,below=4] {Tagged}; 
			\draw (0,0) circle(0) node [black,below=4] {$m_{t-}=4$};
			
			\foreach \x in {1,2,3,4}
			\draw [dashed,black] (\x-0.2,0.2) rectangle (\x+0.2,-0.2)
			[dashed,black] (\x,0) circle (0)
			node [black,above=5]{w};
			
			\foreach \x in {-1,-2}
			\draw [black,fill](\x,0) circle (0.1)
			node [black,above=4]{b};
			
			\foreach \x in {0,-3}
			\draw (\x,0) circle (0.1) node [black,above=4]{p};
			
			\end{tikzpicture}
			
			\begin{tikzpicture}
			\draw[black, thick] (-4.3,0) -- (4.3,0) node[black,right] {$\psi_{t}$}; 
			\draw [black,fill] (-4-0.1,0.1) rectangle (-4+0.1,-0.1) node [black,below=4] {Tagged}; 
			\draw (1,0) circle(0) node [black,below=4] {$m_{t}=5$};
			
			\foreach \x in {2,3,4}
			\draw [dashed,black] (\x-0.2,0.2) rectangle (\x+0.2,-0.2)
			[dashed,black] (\x,0) circle (0)
			node [black,above=5]{w};
			
			\foreach \x in {1,-2}
			\draw [black,fill](\x,0) circle (0.1)
			node [black,above=4]{b};
			
			\foreach \x in {-3,0,-1}
			\draw (\x,0) circle (0.1) node [black,above=4]{p};
			\draw[dashed,<->] (-1,0.7) to (-1,1.1) to (1,1.1)
			to (1,0.7);
			\end{tikzpicture}
			
			\begin{tikzpicture}
			\draw[black, thick] (-4.3,0) -- (4.3,0) node[black,right] {$\psi_{t}$}; 
				\draw [black,fill] (-4-0.1,0.1) rectangle (-4+0.1,-0.1) node [black,below=4] {Tagged}; 
				\draw (-1,0) circle(0) node [black,below=4] {$m_{t}=3$};
				
				\foreach \x in {0,1,2,3,4}
				\draw [dashed,black] (\x-0.2,0.2) rectangle (\x+0.2,-0.2)
				[dashed,black] (\x,0) circle (0)
				node [black,above=5]{w};
				
				\foreach \x in {-2,-3}
				\draw [black,fill](\x,0) circle (0.1)
				node [black,above=4]{b};
				
				\foreach \x in {-1}
				\draw (\x,0) circle (0.1) node [black,above=4]{p};
			\draw[dashed,->]  (-3.5,1.1) to (-3.2,0.4);
			\end{tikzpicture}
			
			\begin{tikzpicture}	\draw[black, thick] (-4.3,0) -- (4.3,0) node[black,right] {$\psi_{t}$}; 
			\draw [black,fill] (-4-0.1,0.1) rectangle (-4+0.1,-0.1) node [black,below=4] {Tagged}; 
			\draw (0,0) circle(0) node [black,below=4] {$m_{t}=4$};
			
			\foreach \x in {1,2,3,4}
			\draw [dashed,black] (\x-0.2,0.2) rectangle (\x+0.2,-0.2)
			[dashed,black] (\x,0) circle (0)
			node [black,above=5]{w};
			
			\foreach \x in {-1,-2}
			\draw [black,fill](\x,0) circle (0.1)
			node [black,above=4]{b};
			
			\foreach \x in {0,-3}
			\draw (\x,0) circle (0.1) node [black,above=4]{p};
			
			\draw[dashed,->]  (-3,0.7) to (-3,1.1) to (-4.2,1.1);
			\end{tikzpicture}
			\caption{Configurations $(\psi,m)$ before and after $\mathcal{N}_{3,5}(t)$ with Realization as a Hole, $\mathcal{C}(t)$ and $\mathcal{D}(t). $ Particularly, $m_{t-}>1. $}
		\end{center}
	\end{figure}
	
	\begin{figure}[h!]
		\begin{center}
			\begin{tikzpicture} 
			\draw[black, thick] (-4.3,0) -- (4.3,0) node[black,right] {$\psi_{t-}$}; 
			\draw [black,fill] (-4-0.1,0.1) rectangle (-4+0.1,-0.1) node [black,below=4] {Tagged}; 
			\draw (-2.7,0) circle(0) node [black,below=7] {$m_{t-}=1$};
			
			\foreach \x in {-2,...,4}
			\draw [dashed,black] (\x-0.2,0.2) rectangle (\x+0.2,-0.2)
			[dashed,black] (\x,0) circle (0)
			node [black,above=5]{w};
			
			\foreach \x in {}
			\draw [black,fill](\x,0) circle (0.1)
			node [black,above=4]{b};
			
			\foreach \x in {-3}
			\draw (\x,0) circle (0.1) node [black,above=4]{p};
			
			\end{tikzpicture}
			
			\begin{tikzpicture}
			\draw[black, thick] (-4.3,0) -- (4.3,0) node[black,right] {$\psi_{t}$}; 
			\draw [black,fill] (-4-0.1,0.1) rectangle (-4+0.1,-0.1) node [black,below=4] {Tagged}; 
			\draw (-2,0) circle(0) node [black,below=7] {$m_{t}=2$};
			
			\foreach \x in {-1,...,4}
			\draw [dashed,black] (\x-0.2,0.2) rectangle (\x+0.2,-0.2)
			[dashed,black] (\x,0) circle (0)
			node [black,above=5]{w};
			
			\foreach \x in {-3}
			\draw [black,fill](\x,0) circle (0.1)
			node [black,above=4]{b};
			
			\foreach \x in {-2}
			\draw (\x,0) circle (0.1) node [black,above=4]{p};
			\draw[dashed,<->] (-3,0.7) to (-3,1.1) to (-2,1.1)
			to (-2,0.7);
			\end{tikzpicture}
			
			\begin{tikzpicture}
			\draw[black, thick] (-4.3,0) -- (4.3,0) node[black,right] {$\psi_{t}$}; 
			\draw [black,fill] (-4-0.1,0.1) rectangle (-4+0.1,-0.1) node [black,below=4] {Tagged}; 
			\draw (-2,0) circle(0) node [black,below=4] {$m_{t}=2$};
			
			\foreach \x in {-1,...,4}
			\draw [dashed,black] (\x-0.2,0.2) rectangle (\x+0.2,-0.2)
			[dashed,black] (\x,0) circle (0)
			node [black,above=5]{w};
			
			\foreach \x in {-2}
			\draw [black,fill](\x,0) circle (0.1)
			node [black,above=4]{b};
			
			\foreach \x in {-3}
			\draw (\x,0) circle (0.1) node [black,above=4]{p};
			\draw[dashed,<->] (-4,0.7) to (-4,1.1) to (-2,1.1)
			to (-2,0.7);
			\end{tikzpicture}
			\caption{Configurations $(\psi,m)$ before and after $\mathcal{N}_{1,2}(t)$,  $
				\mathcal{N}_{0,2}(t)$ with Realizations as Particles. Particularly,$m_{t-}=1.$}
		\end{center}
	\end{figure}
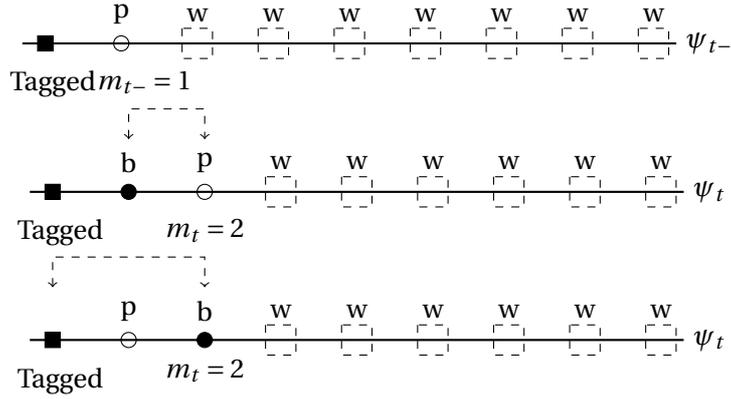	 

 There is a natural regeneration time 
\begin{equation}\label{eq: regeneration time}
\tau= \inf\{s>0: m_s =0\}.
\end{equation} At each $\tau$, there are only white cups with "Bernoulli" particles on positive sites. To have a regenerative structure, we always set $m_{\tau+} = m_{\tau}+1 = 1$, reveal the particle at site 1 with a Bernoulli trial, and color the cup accordingly: $\tilde{c}_{\tau+}(1) = b$ with probability $\rho$, and $\tilde{c}_{\tau+}(1) = p$ with probability $1-\rho$. 

\textcolor{black}{To get estimates of $\tau$ and $X_\tau$, we will consider the stopped process $(\bar{\psi}_t,\bar{m}_t) = (\tilde{c}_{t\wedge \tau}, \tilde{s}_{t\wedge \tau}, m_{t\wedge\tau})$.} It also has a generator $\tilde{L}_{d,2}$. We will not use it for computation, but we can also write it by the graphical construction. From event times at Poisson processes $(\mathcal{N}_{x,y})_{x>y>0}$, $\mathcal{C}$, $\mathcal{D}$, and $\mathcal{N}_{0,2}$, we have  $\tilde{L}_{d,2} =\tilde{S}^{ex}_{+,2} + \tilde{L}^{sh,2} + \tilde{L}^{d,2}$. It acts on a local function $f$ as: 

For $m>0$,
	\begin{equation} \label{eq: L_d,2}
	\tilde{L}_{d,2}f(\psi,m) =\left(\tilde{S}^{ex}_{+,2} + \tilde{L}^{sh,2} + \tilde{L}^{d,2}\right)f(\psi,m)  
	\end{equation}
	\begin{equation}
	\label{eq: new  2 L^sh}
	\tilde{L}^{sh,2} f(\psi,m)= q_1\cdot \mathbb{1}_{\{c(1)=p\}}\left(f(\theta_1\psi,m-1)-f(\psi,m)\right) 
	\end{equation}
	\begin{equation}		
	\label{eq: new 2 L_^d}
	\tilde{L}^{d,2}f(\psi,m) = p_2 \left(f(C_{p,1}\psi,m)-f(\psi,m)\right) 
\end{equation}
\begin{align}
	\label{eq: new  2 S^ex_2}
	\tilde{S}^{ex}_{+,2}f(\psi,m)=&  \sum_{m\geq y>x>0 }p(y-x)\left(f(\psi^{x,y},m)-f(\psi,m)\right)
	\notag\\
	 &+  \sum_{y>m \geq x>0} p(y-x)\sum_{\sigma\in T_{m+1,y}} r(\sigma) \left\lbrace  f\left(\left(\prod_{j=b+1}^y C_{\sigma(j),j}\psi\right)^{x,y},y\right)-f(\psi,m)\right\rbrace
	 \notag\\
	 &+  p_2 \mathbb{1}_{\{m=1\}} \cdot\sum_{\sigma\in T_{2,2}} r(\sigma) \left(  f\left( C_{\sigma(2),2}\psi,2\right)-f(\psi,1)\right)
	 \end{align}
	 where $T_{m,n} = \{b,p\}^{\{m,\dots,n\}}$, $r(\sigma) = \rho^{\sigma_b} (1-\rho)^{\sigma_p}$ and $\sigma_b, \sigma_p$ are the numbers of $b$ and $p$ in $\sigma$.
	  
	  For m= 0,
	  \begin{equation}
	  \tilde{L}_{d,2}f(\psi,0) = 0.
	  \end{equation}
Again, with \textcolor{black}{the graphical construction}, it is easy to see relation \eqref{eq: color and states} holds for the stopped process $(\bar{\psi}_t,\bar{m}_t)$ for any $t\geq 0$ given \eqref{eq: renewal process initial condition}.
	 
For these two auxiliary processes $\eta_t$ and  $(\bar{\psi}_t,\bar{b}_t)$, we denote by $\mathbb{Q}^{\mu_\rho,d,1}$ and $\mathbb{Q}^{\mu_\rho,d,2}$ the corresponding probability measures with initial conditions \eqref{eq: initial condition for Ld,1} and \eqref{eq: renewal process initial condition}, respectively. We denote by $\mathbb{E}^{\mu_\rho,d,1}$ and $ \mathbb{E}^{\mu_\rho,d,2}$ the corresponding expectations. Since relation \eqref{eq: color and states} holds for all $t\geq 0$, we will further identify $\eta_t$ and $ \bar{\psi}_t$ with their color processes $c_t$ and $\tilde{c}_{t\wedge \tau}$.

We end this section by mentioning some connections among the displacement $X_t$, the original environment process $\xi_t$, and auxiliary processes $\zeta_t$, $\eta_t$, $\psi_t$.
\textcolor{black}{
	\begin{remark}\label{rmk: c_t,xi_t}
		 By the graphical construction, we can recover the environment process $(\xi_s)_{s\leq t}$ from the color process $(c_s)_{s\leq t}$, Poisson event times $(\mathcal{N}_{x,y})_{x>y>0}$, $\mathcal{C}$, $\mathcal{D}$, and independent Bernoulli random variables. We can first recover the label process $(l_s)_{s\leq t}$ from the color process $(c_s)_{s\leq t}$ and event times, and then use Bernoulli random variables to reveal white cups which are not revealed by time $t$. A similar approach also works for $\tilde{c}_t$. However, we cannot recover $\xi_t$ from only $\eta_t$ or $\zeta_t$ without knowing the Bernoulli variables and their dynamics. Actually, we don't need to recover $\xi_t$ from $\eta_t$ or $\zeta_t$ because of the followings:
		\begin{enumerate}
		\item The displacement $X_t$ can be written as the number of left shifts of the environment process at event times of $\mathcal{C}$ by time $t$. 
		\begin{equation}\label{eq: displacent representation}
		X_t = \sum_{s\leq t} \mathbb{1}_{\{\xi_{s} = \theta_{1}\xi_{s-},\mathcal{C}(s-)\neq\mathcal{C}(s) \}},\end{equation} 
		which also holds if we replace $\xi_s$ by $\zeta_s$, $\eta_s$ or $\psi_s$. 
		This representation allows us to write different martingales in terms of $\eta_s$ and $\bar{\psi}_s$.
		Examples are a $\mathbb{P}^{\mu_\rho,d}$- martingale $X_t-\int_0^t q_1(1-\xi_s(1)\,ds$, a  $\mathbb{Q}^{\mu_\rho,d,1}$- martingale $X_t-\int_0^t f(\eta_s)\,ds$, and a $\mathbb{Q}^{\mu_\rho,d,1}$- martingale $\exp \left(X_t-\int_0^t (e-1)f(\eta_s)\,ds\right)$, where $f(\eta) = (1-\rho)q_1\mathbb{1}_{\{c(1)=w\}} + q_1\mathbb{1}_{\{c(1)=p\}}$.
		\item The distribution of $\xi_t$ at any fixed time $t$ can be derived from that of $c_t$, see Lemma \ref{lm: Bernoulli particles},
		\begin{align} \label{eq:corresponding distribution}
		\mathbb{P}^{\mu_\rho,d}\left(\xi_t(i)=1, \text{for all $i$ in A }\right) = \sum_{k=0}^{\abs{A}} \rho^k \mathbb{Q}^{\mu_\rho,d,1}\left(\sum_{i\in A} \mathbb{1}_{\{c_t(i)=b\}}=\abs{A}-k, \sum_{i\in A} \mathbb{1}_{\{c_t(i)=w\}}=k \right).
			\end{align} With this, we can get an invariant measure of $\xi_t$ from an invariant measure of $\eta_t$. We will see, for $\eta_t$, we can even derive an ergodic measure and compute its marginal distribution at site 1.
			\item In the proof of LDP for $X_t$, we need to perturb the processes. Although the graphical construction is more intuitive for construction, it contains more information than we need. Particularly, the Radon-Nikodym derivatives, see \eqref{eq: r-n derivative}, are much easier to handle when the reference $\sigma$-algebra is generated by $(\eta_s)_{s\leq t}$ instead of $(\mathcal{N}_{x,y})_{x>y>0}$, $\mathcal{C}$, $\mathcal{D}$ and $\xi_0$.  
		\end{enumerate}
	\end{remark}
}
\section{Law of Large Numbers for $X_t$}\label{sec: LLN for displacement}
We start with the law of large numbers for $X_t$. Consider the auxiliary process $\eta_t$ discussed in section \ref{sec: graphical construction and coloring}. \textcolor{black}{In view of \eqref{eq: displacent representation} in Remark \ref{rmk: c_t,xi_t}, we can get a $\mathbb{Q}^{\mu_\rho,d,1}$- martingale $X_t-\int_0^t f(\eta_s)\,ds$, where $f(\eta)$ is a local function of $\eta$. This martingale has a quadratic variation in time $t$. Once we obtain a law of large numbers for the additive functional $\int_0^t f(\eta_s)\,ds$, we get a law of large numbers for $X_t$. It turns out that for the auxiliary process $\eta_t$, we can get an ergodic measure $\nu_e$ and compute its marginal distribution at site 1 explicitly. This is the main subject of this section. We also identify $\eta_t$ with $c_t$, see \eqref{eq: initial condition for Ld,1}.} 
	
	\textcolor{black}{We should notice that the computation in this section also helps the proof of the large deviation principles of $X_t$, see Lemma \ref{lm:Local LDP at positive good points}. To emphasis the connection between these two sections, we will use a five-tuple $\lambda$ with components $\lambda_{p,D} = q_1$, $\lambda_{w,D} = q_1(1-\rho)$, $\lambda_{w,b}=q_1\rho$, $\lambda_{b,p}=\lambda_{w,p}=p_2$ in the computation involving the generator $\tilde{L}_{d,1}$. Similar estimates hold for every $\lambda$ with positive components.}  

\textcolor{black}{We first obtain candidates for the ergodic measure $\nu_e$ by tightness.} The state space for $\eta_t$ ($c_t$) is ${\tilde{\mathbb{X}} =\{b,p,w\}^{\mathbb{Z}_+}}$, and it is compact with the product topology. By Prokhorov's Theorem, any subset of the space of probability measure $\mathbf{M}_1(\tilde{\mathbb{X}})$ with the weak topology is precompact. By Theorem B7 \cite{Li}, any weak limit $\bar{\nu}$ of the mean of empirical measures $\nu_{t_n}$ is invariant with respect to $\tilde{L}_{d,1}$. Here, $\nu_{t}$ is defined by its action on local functions on $\tilde{\mathbb{X}}$, \[\langle \nu_t,f\rangle:=\frac{1}{t}\mathbb{E}^{\mu_\rho,d,1}[\int_0^t f(\eta_s)\,ds].\]

\textcolor{black}{Next, we get estimates for these weak limits $\bar{\nu}$ from some test functions. There are three estimates, (see Lemmas \ref{lm: priori estimates}, \ref{lm: explicit infor on site 1}, and \ref{lm: finite N_b, N_p}). The first estimate is a priori one, which says that the total number of nonwhite cups is finite in expectation when the cup at site one is not black. There is also an ergodic measure $\nu_e$ satisfying this estimate. The second estimate says that, given the first one, there are balance equations for invariant measures $\bar{\nu}$ and $\nu_e$. From the balance equations, we can compute the marginal distribution at site one explicitly. The third says that the total number of nonwhite cups is finite in mean, which will imply $\nu_e$ is the unique ergodic measure satisfying the above estimates.}

We now consider test functions. At any finite time, there are finitely many non-white cups given $\eta_0$ satisfies condition \eqref{eq: initial condition for Ld,1}. Consider a test function of the form $G_{b,W}(\eta):=\sum_{i>0} W(i)\cdot \mathbb{1}_{\{c(i)=b\}}  $, where $W$ is a nonnegative function on $\mathbb{N}$, with $W(0)=0$. $G_{b,W}(\eta_t)$ is finite since $c_t(i) = w$ for $i$ large.  We can compute $\tilde{S}^{ex}_+ G_{b,W}$, $\tilde{L}^{sh} G_{b,W}$, $\tilde{L}^{d}G_{b,W}$ with summation by parts,
\begin{align}\label{eq: S^ex_+ for linear weighted sum}
	\tilde{S}^{ex}_+ G_{b,W} &= \sum_{i>0} W(i)\cdot  \sum_{j>0} p(j-i) \left( \mathbb{1}_{\{c(j)=b\}}-\mathbb{1}_{\{c(i)=b\}}\right) \notag \\
	&= \sum_{i>0} (\Delta_{p,+} W)(i) \cdot \mathbb{1}_{\{c(i)=b\}}
\end{align} where $(\Delta_{p,+} W)(i) = \sum_{y> -i} p(y) \left(W(i+y) -W(i)\right) $.
\begin{align}\label{eq: L^sh_+ for linear weighted sum}
\tilde{L}^{sh} G_{b,W} =&\left(\lambda_{p,D}\mathbb{1}_{\{c(1)=p\}}+ \lambda_{w,D}\mathbb{1}_{\{c(1)=w\}} \right)  \sum_{i>0} \left(\mathbb{1}_{\{c(i+1)=b\}}-\mathbb{1}_{\{c(i)=b\}} \right)W(i)
\notag \\
&+ \lambda_{w,b}\mathbb{1}_{\{c(1)=w\}} W(1)
\notag \\
=& \left(\lambda_{p,D}\mathbb{1}_{\{c(1)=p\}}+ \lambda_{w,D}\mathbb{1}_{\{c(1)=w\}} \right)  \sum_{i>1} (\nabla_{-1} W)(i)\cdot \mathbb{1}_{\{c(i)=b\}}
\notag \\
&+ \lambda_{w,b}\mathbb{1}_{\{c(1)=w\}} W(1),
\end{align} 
where $(\nabla_{-1} W)(i) = W(i-1)-W(i) $.
\begin{align}\label{eq: L^d for linear weighted sum}
\tilde{L}^{d} G_{b,W} = -\lambda_{b,p}\mathbb{1}_{\{c(1)=b\}} W(1).
\end{align}
Combining \eqref{eq: S^ex_+ for linear weighted sum},\eqref{eq: L^sh_+ for linear weighted sum}, and \eqref{eq: L^d for linear weighted sum}, we have $\mathit{L}_{d,1}G_{b,W}$ as:
\begin{align} \label{eq: L_d,1 acts on G_b,w}
	\tilde{L}_{d,1}G_{b,W} =& \sum_{i>0} (\Delta_{p,+} W)(i) \cdot \mathbb{1}_{\{c(i)=b\}} \notag \\
	&+ \left( \lambda_{p,D}\mathbb{1}_{\{c(1)=p\}} +\lambda_{w,D} \mathbb{1}_{\{c(1)=w\}} \right) \cdot\sum_{i>1} (\nabla_{-1} W)(i)\cdot \mathbb{1}_{\{c(i)=b\}}
	\notag \\ 
	&+ \lambda_{w,b}\mathbb{1}_{\{c(1)=w\}}W(1)  -\lambda_{b,p}\mathbb{1}_{\{c(1)=b\}} W(1).
\end{align}
Similarly, for $G_{p,W}:= \sum_{i>0} W(i)\mathbb{1}_{\{c(1)=p\}}$, we have $\tilde{L}_{d,1}G_{p,W}$ as:
\begin{align} \label{eq: L_d,1 acts on G_p,w}
\tilde{L}_{d,1}G_{p,W} =& \sum_{i>0} (\Delta_{p,+} W)(i) \cdot \mathbb{1}_{\{c(i)=p\}} \notag \\
+& \left( \lambda_{p,D}\mathbb{1}_{\{c(1)=p\}} + \lambda_{w,D} \mathbb{1}_{\{c(1)=w\}} \right) \cdot\sum_{i>1} (\nabla_{-1} W)(i)\cdot \mathbb{1}_{\{c(i)=b\}}
\notag \\ 
+& \left(\lambda_{w,p}\mathbb{1}_{\{c(1)= w\}} + \lambda_{b,p}\mathbb{1}_{\{c(1)= b\}}\right) W(1)  -\lambda_{p,D}\mathbb{1}_{\{c(1)=p\}} W(1).
\end{align}
Equations \eqref{eq: L_d,1 acts on G_b,w} and \eqref{eq: L_d,1 acts on G_p,w} will be applied with some choices of $W(i)$. Also, for $j = b \text{ or }p$, we denote by $\langle \nu, G_{j,w}\rangle$ the limit
 \begin{equation} \label{eq: right sense nonlocal function}
 	\langle \nu, G_{j,W}\rangle := \sup_k \langle \nu, G_{j,W,k}\rangle = \lim_{k\to\infty} \langle \nu, \sum_{i=1}^k W(i)\cdot \mathbb{1}_{\{c(i)=j\}}\rangle.
 \end{equation}

\textcolor{black}{Formally, we want to have $(\Delta_{p,+} W)(i)=0$ for $i$ large enough.} Consider the first four quantities $N_b$, $N_p$, $W_b$, $W_p$, when we choose $W(i) = 1 \text{ or } i$, for $i>0$,
\begin{align}
	\label{eq: N_b,N_p}	
	N_b =& \sum_{i>0} \mathbb{1}_{\{c(i)=b\}}, \quad 	N_p = \sum_{i>0} \mathbb{1}_{\{c(i)=p\}},
	\\
	\label{eq: W_b,W_p}
	W_b =& \sum_{i>0} i\cdot \mathbb{1}_{\{c(i)=b\}}, \quad 
	W_p = \sum_{i>0} i\cdot \mathbb{1}_{\{c(i)=p\}}.
\end{align} Particularly, the first two are the total numbers of blue cups and purple cups. 

The first lemma says that $\mathbb{1}_{\{c(1)=p\}}\cdot N_b$ ,$\mathbb{1}_{\{c(1)=w\}}\cdot N_b$, $\mathbb{1}_{\{c(1)=p\}}\cdot N_p$ and $\mathbb{1}_{\{c(1)=w\}}\cdot N_p$ are all uniformly bounded in expectation with respect to $\nu_t$ and any weak limit $\bar{\nu}$:
\begin{lemma} \label{lm: priori estimates}
	Consider the auxiliary process $\eta_t$ with initial condition \eqref{eq: initial condition for Ld,1}. Let $\nu_t$ be the mean of empirical measures at time $t$, and $\bar{\nu}$ be a weak limit of any subsequence $\nu_{t_n}$ (if it exists). Then, there is a positive constant $C>0$ such that for all $t \geq 0$, $j = b$ or $p$, $\nu = \nu_t$ or $\bar{\nu}$,
	\begin{equation}\label{eq: priori estimate}
		\langle \nu, \mathbb{1}_{\{c(1)=p\}}\cdot N_j \rangle,\quad \langle \nu, \mathbb{1}_{\{c(1)=w\}}\cdot N_j \rangle \leq C.
	\end{equation}
	We understand the above notions in the sense of \eqref{eq: right sense nonlocal function}.
\end{lemma}  
\begin{proof}:
	We will show the case when $j=b$ since the other case follows similar arguments. Consider $W_b(t) = G_{b,W}(\eta_t) =\sum_{i>0} i\cdot \mathbb{1}_{\{c_t(i)=b\}}$, which is finite at any time $t\geq 0$. Applying Ito's formula, we have a $\mathbb{Q}^{\mu_\rho,d,1}$- martingale $W_b(t) - \int_{0}^{t} \tilde{L}_{d,1} W_b(s) \,ds$. By equation  \eqref{eq: L_d,1 acts on G_b,w},
	\begin{align}
		\tilde{L}_{d,1} W_b =& - \left( \lambda_{p,D}\mathbb{1}_{\{c(1)=p\}} + \lambda_{w,D} \mathbb{1}_{\{c(1)=w\}} \right)\cdot N_b
		\notag \\ &+ \sum_{k=1,2} a_k \mathbb{1}_{\{c(k)=w\}} + b_k\mathbb{1}_{\{c(k)=b\}}, 
	\end{align} 
	where $a_k, b_k$ \textcolor{black}{are constants depending on $\lambda$.} Taking expectation with respect to $\mathbb{Q}^{\mu_\rho,d,1}$, we have
	\begin{align}
		&\mathbb{E}^{\mu_\rho,d,1}\left[\int_{0}^{t}\left( \lambda_{p,D}\mathbb{1}_{\{c(1)=p\}} +\lambda_{w,D} \mathbb{1}_{\{c(1)=w\}} \right)\cdot N_b(s) \,ds\right]  \notag \\
		=& \mathbb{E}^{\mu_\rho,d,1}\left[\int_{0}^{t} \sum_{k=1,2} a_k \mathbb{1}_{\{c_s(k)=w\}} + b_k\mathbb{1}_{\{c_s(k)=b\}}\,ds \right] -\mathbb{E}^{\mu_\rho,d,1}\left[W_b(t)\right]. 
	\end{align}
	Since $W_b\ge 0$, dividing $t$ on both sides, we see
	\begin{equation}\label{eq: important estimate}
		 \langle \nu_t,  \left( \lambda_{p,D}\mathbb{1}_{\{c(1)=p\}} +\lambda_{w,D} \mathbb{1}_{\{c(1)=w\}} \right)\cdot N_b  \rangle \leq \sum_{k=1}^2 \abs{a_k} +\abs{b_k},
	\end{equation}
	which is sufficient for \eqref{eq: priori estimate}, \textcolor{black}{as $\lambda_{p,D}$ and $\lambda_{w,D}$ are strictly positive}.
\end{proof}	

	Consider $\tilde{L}_{d,1} N_b$ and $\tilde{L}_{d,1} N_p$. By equations \eqref{eq: L_d,1 acts on G_b,w} and \eqref{eq: L_d,1 acts on G_p,w}, we have
	\begin{align*}
	\tilde{L}_{d,1} N_b =& \lambda_{w,b}\mathbb{1}_{\{c(1)=w\}}  -\lambda_{b,p}\mathbb{1}_{\{c(1)=b\}}   \\
	\tilde{L}_{d,1} N_p =& \lambda_{w,p}\mathbb{1}_{\{c(1)=w\}}+\lambda_{b,p}\mathbb{1}_{\{c(1)= b\}}  -\lambda_{p,D}\mathbb{1}_{\{c(1)=p\}}.
	\end{align*}
	Taking expectation with respect to some invariant measure $\bar{\nu}$, together with $\bar{\nu}_b+\bar{\nu}_p+\bar{\nu}_c=1$, we expect (formal) balance equations of $\bar{\nu}_b=\langle \bar{\nu},\mathbb{1}_{\{c(1)=b\}}\rangle$, $\bar{\nu}_p=\langle \bar{\nu},\mathbb{1}_{\{c(1)=p\}}\rangle$,$\bar{\nu}_w=\langle \bar{\nu},\mathbb{1}_{\{c(1)=w\}}\rangle$:
	\begin{equation} \label{eq: marginal equations}
	\begin{cases}
	 \lambda_{w,b}\bar{\nu}_w  -\lambda_{b,p}\bar{\nu}_b   &= 0\\
	 \lambda_{w,p}\bar{\nu}_w + \lambda_{b,p}\bar{\nu}_b  -\lambda_{p,D}\bar{\nu}_p &= 0
	\\
	\bar{\nu}_w +\bar{\nu}_b+\bar{\nu}_p &= 1
	\end{cases}
	.\end{equation}
	This is a formal computation because neither $N_b$ nor $N_p$ is a local function. We can approximate $N_b$ and $N_p$ by functions $N_{b,r}=G_{b,W}$ and  $N_{p,r}=G_{p,W}$, where $W$ has geometric weights $W(i)=r^i$, for some $r<1$. These are almost local functions, and we can get that $\langle\bar{\nu},\tilde{L}_{d,1}N_{b,r}\rangle=0$. The estimate \eqref{eq: priori estimate} allows us to take limit as $r$ goes to $1$ and derive \eqref{eq: marginal equations}. We have the second lemma.
	
	\begin{lemma}\label{lm: explicit infor on site 1}
		Consider the auxiliary process $\eta_t$ with the generator $\tilde{L}_{d,1}$, and let $\bar{\nu}$ be some invariant measure with respect to $\tilde{L}_{d,1}$. If there is a constant $C>0$, such that $\bar{\nu}$ satisfies \eqref{eq: priori estimate}, we have the marginal distribution of site $1$ solves equations 
		\eqref{eq: marginal equations}. Particularly,	\begin{align}
		\label{eq: marginal density for d_b}
		\bar{\nu}_b
		=\frac{\rho q^2_1}{\left(q_1 + p_2\right) \left(p_2 + \rho q_1 \right)}, 
		\bar{\nu}_p=
		\frac{p_2}{q_1+p_2},
		\bar{\nu}_w=
		\frac{p_2 q_1}{\left(q_1+p_2 \right) \left(p_2 + \rho q_1 \right)}. 
		\end{align}
		
	
	\end{lemma}
	\begin{proof}: We will show the first equation of \eqref{eq: marginal equations}, and the second follows a similar argument.
		
	Consider $W(i) = r^i$,for $ i>0$. Let $N_{b,r} = G_{b,W} = \sum_{i>0} r^i \mathbb{1}_{\{c(i)=b\}}$, which is bounded by $(1-r)^{-1}$ for some $r \in (0,1)$. $\lim_{r\to 1}N_{b,r} = N_b$. \textcolor{black}{By the invariance of $\bar{\nu}$, $\langle \bar{\nu},\tilde{L}_{d,1} N_{b,r}\rangle =0$. For this equality, one can approximate $N_{b,r}$ by local functions $N_{b,r,k}= \sum_{0<i<k} r^i \mathbb{1}_{\{c(i)=b\}}$, and use the invariance of $\bar{\nu}$ to get $\langle \bar{\nu},\tilde{L}_{d,1} N_{b,r,k}\rangle =0$. Taking limit as $k$ goes to infinity, we will get the equality. On the other hand, when we want to take limit as $r$ goes to $1$, we need to take care of the singularity at 1.}
	
	Notice that for $i\geq 3$, 
	\begin{equation*}
		(\Delta_{p,+} W)(i) =  (1-r)^2\cdot g_p(r)W(i),
	\end{equation*} where $g_p(r)$ is a rational function of $r$ involving $p(\cdot)$ with one singularity at $0$, and 
	\[(\nabla_{-1} W)(i) = W(i) (1-r)r^{-1}, \text{ for } i\geq 2.\]
	
	By equation \eqref{eq: L_d,1 acts on G_b,w},
	\begin{align*} 
	\tilde{L}_{d,1}N_{b,r} =& \sum_{i>0} (\Delta_{p,+} W)(i) \cdot \mathbb{1}_{\{c(i)=b\}} 
	\notag \\ 
	&+ \left( \lambda_{p,D}\mathbb{1}_{\{c(1)=p\}} +\lambda_{w,D} \mathbb{1}_{\{c(1)=w\}} \right) \cdot\sum_{i>1} (\nabla_{-1} W)(i)\cdot \mathbb{1}_{\{c(i)=b\}}
	\notag \\ 
	&+ \lambda_{w,b}\mathbb{1}_{\{c(1)=w\}} W(1)  -\lambda_{b,p}\mathbb{1}_{\{c(1)=b\}} W(1) 
	\\ 
	=& (1-r)^2\cdot g_p(r) N_{b,r} -(1-r)\sum_{i=1}^2 h_{i,p}(r) \mathbb{1}_{\{c(i)=b\}}
	\\ 
	&+(1-r)r^{-1}\left( \lambda_{p,D}\mathbb{1}_{\{c(1)=p\}} +\lambda_{w,D} \mathbb{1}_{\{c(1)=w\}} \right)\cdot \left(N_{b,r} - r\cdot \mathbb{1}_{\{c(1)=b\}}\right)
	\\ 
	&+ r\left(\lambda_{w,b}\mathbb{1}_{\{c(1)=w\}}   -\lambda_{b,p} \mathbb{1}_{\{c(1)=b\}}\right), 
	\end{align*} where $h_{i,p}(r)$ is also a rational function on $r$ with a singularity at $0$, \textcolor{black}{it also involves components of $\lambda$}.  
	Taking expectation with respect to $\bar{\nu}$, and take limit as $r \uparrow 1$, we see that only the last line remains. Since $N_{b,r}$ is uniformly bounded by $(1-r)^{-1}$ and $N_{b,r} \leq N_b$, with estimates \eqref{eq: priori estimate}, we can apply Dominated Convergence Theorem to get rid of the first two lines. That is, 
	\begin{equation}
		\lim_{r\uparrow1 } \langle \bar{\nu},\tilde{L}_{d,1} N_{b,r}\rangle = \lambda_{w,b}\bar{\nu}_w   -\lambda_{b,p} \bar{\nu}_b = 0.
	\end{equation}
	Solving equations \eqref{eq: marginal equations}, we get \eqref{eq: marginal density for d_b}.	
	\end{proof}
	
	The third lemma says that if an invariant measure $\bar{\nu}$ with respect to $\tilde{L}_{d,1}$ satisfies estimates \eqref{eq: priori estimate}, $N_b$ and $N_p$ are both in $L_1(\bar{\nu})$. \textcolor{black}{This is a consequence of removals of particles, $\lambda_{b,p}>0$.}
	\begin{lemma}\label{lm: finite N_b, N_p}
		Under the assumptions of Lemma \ref{lm: explicit infor on site 1}, if there is a constant $C>0$, such that $\bar{\nu}$ satisfies \eqref{eq: priori estimate}, we further have, for some constant $C_1$ depending on $C$, $q_1, p(\cdot)$, and $\rho$:
		\begin{equation} \label{eq: estimate for N_b, N_q}
			\langle \bar{\nu},N_b \rangle, \langle \bar{\nu},N_p \rangle < C_1.
		\end{equation}
	\end{lemma}
	\begin{proof}: We only show the first one. It is similar to the proof of Lemma \ref{lm: explicit infor on site 1}.
		
		Let $W(i) = r^i$ for $i >L \geq 5$ and $W(i)=0$ otherwise. Consider $N_{b,r,L}:= G_{b,W}$, \textcolor{black}{which is bounded by  $(1-r)^{-1}$}. We then compute $\tilde{L}_{d,1}\left(\mathbb{1}_{\{c(1)=b\}}N_{b,r,L}\right)$. Since $\tilde{S}^{ex}_+$ and $\tilde{L}^d$ only involve interchanges of sites, and $N_{b,r,L}$ depends on sites far from site 1,  we can apply the product rule:
		\[(\tilde{S}^{ex}_+ + \tilde{L}^d)(\mathbb{1}_{\{c(1)=b\}}\cdot N_{b,r,L}) = (\tilde{S}^{ex}_+ + \tilde{L}^d)\mathbb{1}_{\{c(1)=b\}} \cdot N_{b,r,L} + (\tilde{S}^{ex}_+ + \tilde{L}^d)N_{b,r,L} \cdot  \mathbb{1}_{\{c(1)=b\}}.\] Therefore, applying equation \eqref{eq: S^ex_+ for linear weighted sum} and \eqref{eq: L^d for linear weighted sum}, we have 
		\begin{align}
			&(\tilde{S}^{ex}_+ + \tilde{L}^d)(\mathbb{1}_{\{c(1)=b\}}\cdot N_{b,r,L})
			\notag \\ 
			=&  \left(p_2 \mathbb{1}_{\{c(3)=b\}} + p_1 \mathbb{1}_{\{c(2)=b\}} - \left(p_1+ p_2 + \lambda_{b,p}\right)\mathbb{1}_{\{c(1)=b\}} \right) \cdot N_{b,r,L}
			\notag \\ 
			&+ (1-r)^2 \cdot g_p(r)\mathbb{1}_{\{c(1)=b\}} \cdot N_{b,r,L} 
			\notag \\ 
			&+ T_3 
			\notag \\ 
			=& -\left( p_1+p_2 +\lambda_{b,p} \right)\cdot N_{b,r,L} +\left(p_2 \mathbb{1}_{\{c(3)=b\}} + p_1 \mathbb{1}_{\{c(2)=b\}}\right)\cdot N_{b,r,L}
			\notag \\ 
			&+ (1-r)^2 \cdot g_p(r)\cdot N_{b,r,L}\cdot \mathbb{1}_{\{c(1)=b\}} 
			\notag \\ 
			&+ T_3 + \left( p_1+p_2 +\lambda_{b,p} \right)\mathbb{1}_{\{c(1)\neq b\}}\cdot N_{b,r,L} \notag
		\end{align} where $T_3$ is a bounded boundary term which involves finitely many  $\mathbb{1}_{\{c(i)=b\}}, \mathbb{1}_{\{c(i)=p\}}$ and \textcolor{black}{$\lambda$}. 
		On the other hand, taking expectation with respect to an $\tilde{L}_{d,1}$- invariant measure $\bar{\nu}$, and rearranging terms, we have
		\begin{align} \label{eq: N_b,r equation}
			& \left( p_1+p_2+\lambda_{b,p}\right)\langle \bar{\nu}, N_{b,r,L}\rangle  -\left(p_2 \langle \bar{\nu},\mathbb{1}_{\{c(3)=b\}}\cdot N_{b,r,L}\rangle + p_1 \langle \bar{\nu},\mathbb{1}_{\{c(2)=b\}}\cdot N_{b,r,L}\rangle \right) 
			\notag \\
			=& (1-r)^2 \cdot g_p(r) \langle \bar{\nu}, \mathbb{1}_{\{c(1)=b\}} \cdot N_{b,r,L} \rangle
			+ \left( p_1+p_2+\lambda_{b,p} \right)\langle \bar{\nu},\mathbb{1}_{\{c(1)\neq b\}}\cdot N_{b,r,L} \rangle 
			\notag \\ 
		 &+ \langle  \bar{\nu}, T_3 \rangle + \langle \bar{\nu},\tilde{L}^{sh}\left(\mathbb{1}_{\{c(1)=b\}}\cdot N_{b,r,L}\right) \rangle
		\end{align}
		By estimates \eqref{eq: priori estimate}, the right hand side of equation \eqref{eq: N_b,r equation} is uniformly bounded in $r \in (0,1)$, and the left hand side is bounded below by $\lambda_{b,p} \langle \bar{\nu}, N_{b,r,L} \rangle$. Therefore, taking limit as $r\uparrow 1$, by Monotone Convergence Theorem, we have 
		\begin{equation}
			\lambda_{b,p} \langle \bar{\nu}, N_{b,1,L} \rangle \leq C_2
		\end{equation} which is sufficient for \eqref{eq: estimate for N_b, N_q}.
	\end{proof}
	
	Now we can prove Theorem \ref{Thm: LLN for X_t}. We first see the existence of an ergodic measure $\nu_e$ with respect to $\tilde{L}_{d,1}$, and $\nu_e$ satisfies estimates  \eqref{eq: priori estimate} and \eqref{eq: estimate for N_b, N_q}. As a consequence of the graphical construction and estimate \eqref{eq: estimate for N_b, N_q}, there is a positive probability to get rid of all nonwhite cups within finite time. Then, we obtain a law of large numbers for additive functionals when the initial configuration has only white cups. As a consequence, we get Theorem \ref{Thm: LLN for X_t}.

	\begin{proof}(Theorem \ref{Thm: LLN for X_t}): \textcolor{black}{By Lemma \ref{lm: priori estimates}, there exists a measure $\bar{\nu}$, which is invariant with respect to $\tilde{L}^{d,1}$, and it satisfies \[\langle\bar{\nu}, \mathbb{1}_{\{c(1)\neq b\}} \cdot(N_b+N_p) \rangle \leq C,\]for some $C>0$. Since $\bar{v}$ can be written as convex combinations of ergodic measures with respect to $\tilde{L}^{d,1}$, there is an ergodic measure $\nu_e$ satisying the above estimate (in the sense of \eqref{eq: right sense nonlocal function}). By Lemma \ref{lm: explicit infor on site 1}, \ref{lm: finite N_b, N_p}, we get the marginal distribution of $\nu_e$ at site 1, see \eqref{eq: marginal density for d_b}, and the expectation of $N_b+N_p$ to be finite, see \eqref{eq: estimate for N_b, N_q}. }
		 
		 \textcolor{black}{
		 For the process $\eta_t$ starting from the ergodic measure ${\nu}_e$, we have an ergodic theorem: for any local function on $\tilde{\mathbb{X}} =\{b,p,w\}^{\mathbb{Z}_+}$,
		 \begin{equation}\lim_{t\to\infty} \frac{1}{t}\int_0^t f(\eta_s)\,ds = \langle \nu_e,f \rangle, \quad \quad   \mathbb{Q}^{\nu_e,d,1}-\text{a.s.} \label{eq: ergodic theorem0}
		 \end{equation}} On the other hand, by \eqref{eq: estimate for N_b, N_q}, there is a large $L>0$, such that \[\mathbb{Q}^{\nu_e,d,1}\left(c_0(i)=w, \text{ for all } i > L\right) >0 \] By the graphical construction, for a fixed time $t_0>0$, we can remove all nonwhite cups, and 
		 \begin{equation} \label{eq: positive "renew" probability}
		 \mathbb{Q}^{\nu_e,d,1}\left(c_{t_0}(i)=w, \text{ for all } i >0 \right) >0 
		 \end{equation} Therefore, by the Markov Property of $\mathbb{Q}^{\nu_e,d,1}$, \textcolor{black}{we get a limit theorem for the initial distribution $\mu_\rho$ from \eqref{eq: ergodic theorem0}: for any local function $f$ on $\tilde{\mathbb{X}} =\{b,p,w\}^{\mathbb{Z}_+}$,
%
		 \begin{equation}
		 	\lim_{t\to\infty} \frac{1}{t}\int_0^t f(\xi_s)\,ds = \langle\nu_e,f \rangle, \quad \quad   \mathbb{Q}^{\bar{\mu}_\rho,d,1}-\text{a.s.} \label{eq: ergodic theorem}
		 \end{equation} This also implies the convergence of $\nu_t$, $\lim_{t\to\infty}\nu_t = \nu_e $.}
As ${X_t - \int_{0}^{t}q_1 \mathbb{1}_{\{c_s(1)=p\}} + \rho q_1 \mathbb{1}_{\{c_s(1)=w\}} \,ds}$ is a $\mathbb{Q}^{\mu_\rho,d,1}$-martingale with a quadratic variation of order $t$, we get \eqref{Eq: LLN for X_t} from \eqref{eq: ergodic theorem} and \eqref{eq: marginal density for d_b},
		\begin{equation}\label{eq: LLN for jumps}
		\lim_{t\to\infty} \frac{X_t}{t} = q_1 \bar{\nu}_p + \rho q_1 \bar{\nu}_w = \left(\frac{1}{q_1}+\frac{\rho}{p_2}\right)^{-1},  \quad \mathbb{Q}^{\bar{\mu}_\rho,d,1}-\text{a.s.}
		\end{equation}
	\end{proof}
	
	We end this section by some remarks on the ergodic measure $\nu_e$ for $\eta_t$ and the corresponding invariant measure $\mu$ for the environment process $\xi_t$.
	\begin{remark}\label{rm: ergodic measure}
		\begin{enumerate}
			\item \textcolor{black}{The estimates \eqref{eq: priori estimate} and \eqref{eq: estimate for N_b, N_q} are important for the ergodic measure $\nu_e$. Under these estimates, we can get both existence and uniqueness of the ergodic measure even under perturbation. Since these two estimates will be finite when jump rates $\lambda= \{\lambda_i\}$ are positive, we obtain an ergodic measure for the perturbed color process by similar arguments in the proof of Theorem \ref{Thm: LLN for X_t}. This will be used in the proof of LDP in the section \ref{sec: LDP}, see Lemma \ref{lm:Local LDP at positive good points}.}
			
			\item From estimate \eqref{eq: estimate for N_b, N_q}, the ergodic measure $\nu_e$ is supported on the set of configurations with finitely many non-white cups. This is a countable subset of ${\tilde{\mathbb{X}} =\{b,p,w\}^{\mathbb{Z}_+}}$. From $\nu_e$, we have positive recurrence of $\eta_t$ or $c_t$ and a natural regenerative structure.
%
			\item In view of Remark \ref{rmk: c_t,xi_t}, there is an invariant measure $\mu$ for the environment process $\xi_t$ corresponding to the ergodic measure $\nu_e$. Since $\nu_e$ has finitely many nonwhite cups, $\mu$ is close to a Bernoulli measure $\mu_\rho$: for any finite subset $A\subset \mathbb{N}$,
			\[\lim_{i\to\infty}\langle\mu,\xi_{A+i}\rangle = \lim_{i\to\infty}\langle \nu_e, \prod_{j\in i+A}\mathbb{1}_{\{c(j)=w \}}\rangle = \rho^{\abs{A}}. \]
			In fact, in the (only) special case when $q_1= (1-\rho)p_2$, we can get $\mu =\mu_\rho$ by using \eqref{eq: ergodic theorem} and verifying $\langle\mu_\rho,\mathit{L}_d\xi_A\rangle =0$ for all finite subset $A \subset \mathbb{N}$. For other values of $\rho$, we see $\mu \neq \mu_\rho$ from $m\neq (1-\rho)q_1$.
			
			\item We can deduce that $\mu$ is ergodic for $\xi_t$ by arguments in this section. In deed, we can construct an auxiliary process $\eta_{n,t}$ by coloring cups on the interval $(1,n)$. By analogues of Lemma \ref{lm: priori estimates}, \ref{lm: finite N_b, N_p} and proof of Theorem \ref{thm: LDP for displacemt}, we get an ergodic measure $\nu_{e;n}$ for $\eta_{n,t}$. For a local function $f$ depending on $\{\xi(i): i\leq n\}$, there is a 
			local function $\tilde{f}$ under $\nu_{e;n}$ such that			$f(\xi_t)=\tilde{f}(\eta_{n,t}),$ from which it follows that
			 \[\lim_{t\to\infty}\frac{1}{t}\int_{0}^{t}f(\xi_s)\,ds = \langle\nu_{e;n}, \tilde{f}\rangle=\langle\mu, f\rangle,\quad \mathbb{P}^{\mu_\rho,d} -\text{a.s.} \]
			 In view of the third point in this remark, we have ergodic measures $\mu$ for $\xi_t$ corresponding to different values of $\rho$, and (only) when $q_1 = (1-\rho)p_2$, the Bernoulli measure $\mu_\rho$ is ergodic.
		\end{enumerate} 
	\end{remark}
\section{Regenerative Structure and Functional Central Limit Theorem}\label{sec: Regenerative Structure}
\textcolor{black}{ We want to prove that the system is strongly recurrent. This is achieved by proving that the regeneration time, which is the return time
	to some convenient configuration, has a finite second moment. This is a classical approach, but its employment depends on models and is often different. For applications of regenerative structure in similar models, we refer to \cite{BR,CQR07,CQR09,JMR} and other references mentioned in the introduction.}
	
	 In our model, we will use the second auxiliary process $(\bar{\psi}_t,\bar{m}_t)$ from section \ref{sec: graphical construction and coloring} and estimate regeneration time $\tau$ and $X_\tau$. We observe the followings. The regeneration time $\tau$ is the return time that the system has only white cups and, equivalently, $\bar{m}_t=0$. The boundary process $\bar{m}_t$ increases when new particles are revealed at event times of $(\mathcal{N}_{x,y})_{x>y}$. It increases by 1 at a rate of $p_1+p_2$ and by 2 at a rate of $p_2$. Meanwhile, $\bar{m}_t$ decreases only when $X_t$ increases. In view of the law of large numbers of $X_t$, $\bar{m}_t$ decreases at a rate $m$. We expect a regenerative time with nice moments when $m > w = 3p_2 +p_1$.
	
	\textcolor{black}{
	On the other hand, with the second auxiliary process, we can also interpret the rate $m$ as the inverse of the average occupation time of a cup at site 1 by time $\tau$. A white cup is colored blue or purple according to a Bernoulli random variable when $m_t$ increases. The removal of this cup always occurs at event times of $\mathcal{C}$ when the cup is purple and is at site 1. If this cup is colored blue initially, it will be colored purple at site 1 at event times of $\mathcal{D}$ before its removal. As $\mathcal{C}$, $\mathcal{D}$ are Poisson processes, we expect the occupation time of a purple cup at site 1 by time $\tau$ to be distributed as an exponential random variable with a parameter $q_1$, and the occupation time of a blue cup at site 1 by time $\tau$ (without the time it is colored purple) to be distributed as an exponential random variables with parameters $p_2$. By averaging, we expect $m= (q^{-1}_1 + \rho p_2^{-1})^{-1}$.} 

\textcolor{black}{
We will carry out the computation more rigorously and also derive moment generating functions of $\tau$ and $m_\tau$ explicitly. We will consider four numbers corresponding to changes of colors by time $t$, and derive some exponential martingales. By applying the Optional-Stopping Theorem at time $\tau$ and choosing parameters, we will get the moment generating function of $\tau$ for negative values near $0$. These parameters can be chosen via the occupation time described before. And the condition $w<m$ allows us to extend the moment generating function to positive values analytically. }

\textcolor{black}
{Consider four processes $N_{w,b}$, $N_{w,p}$, $N_{b,p}$ and $N_{p,D}$, which count the numbers of successful color changes at event times of $(\mathcal{N}_{x,y})_{x>y}$, $\mathcal{C}$ and $\mathcal{D}$. These changes can be tracked by the dynamics of $(\bar{\psi}_t,\bar{m}_t)$. Assume $t\leq \tau$, $\bar{m}_t \geq 1$. $N_{w,b}$ counts the number of white cups colored blue and it increases only when $\bar{m}_t$ increases; $N_{w,p}$ counts the number of white cups colored purple and it also increases only when $\bar{m}_t$ increases; $N_{b,p}$ counts the number of blue cups colored purple, and it increases only at event times $t$ of $\mathcal{D}$ with $\tilde{c}_{t-}(1) =b$;  $N_{p,D}$ counts the number of times $\bar{m}_t$ decreases, and it only increases at event times $t$ of $\mathcal{D}$ with $\tilde{c}_{t-}(1) =p$. They are jump processes with varying jump rates depending on $\tilde{c}_t$. By the graphical construction, we can write 
 exponential martingales associated to the jump processes, and see that these processes are almost orthogonal:}
 
 \begin{lemma}\label{lm: orthogonal exponential martingales}
 	Consider the stopped process $(\bar{\psi}_t,\bar{m}_t)$ with initial condition \eqref{eq: renewal process initial condition}. For $a,b,c,d \in \mathbb{R}$, we have exponential martingales $M_t(a,b,c,d)$ as 
 	\begin{align} \label{eq: exponential martingale}
 		M_t(a,b,c,d) = \exp&\left[a \left(N_{w,b}(t) -N_{w,b}(0)\right)  + b\left(N_{w,p}(t) -N_{w,p}(0)\right) \right. 
 		\notag \\ 
 		& +c\left(N_{b,p}(t) -N_{b,p}(0)\right) +d\left(N_{p,D}(t) -N_{p,D}(0)\right) 
 		\notag \\ 
 		&- \int_{0}^{t\wedge\tau}\mathbb{1}_{\{\bar{m}_s\geq 1\}}\cdot (p_2+p_1) \left(\rho\exp(a) + (1-\rho)\exp(b)-1)\right) 
 		\notag\\
 		&+ \mathbb{1}_{\{\bar{m}_s\geq 1\}}\cdot p_2 \left((\rho\exp(a)+(1-\rho)\exp(b))^2-1\right)
 		\notag\\ 
 		&+ \mathbb{1}_{\{c_s(1)=b\}}\cdot p_2 \left(\exp(c)-1)\right)
 		\notag \\ 
 			&+ \left. \mathbb{1}_{\{c_s(1)=p\}}\cdot q_1 \left(\exp(d)-1)\right)\,ds \right]
 	\end{align}
 \end{lemma}
 \begin{proof}: These martingales are similar to the exponential martingale in the first point of Remark \ref{rmk: c_t,xi_t}. They correspond to exponential martingales obtained from counting processes with varying jump rates depending on $\bar{\psi}_t$ and $\bar{m}_t$. These martingales should be compared with exponential martingales obtained from the formula, see Appendix 1.7 \cite{KL},
	 \begin{align*}
	 \mathbb{M}^{F}_t= \exp \left\{F(\bar{\psi}_t,\bar{m}_t)-F(\bar{\psi}_0,\bar{m}_0) 
	 - \int_{0}^{t}\,ds\, e^{-F(\bar{\psi}_s,\bar{m}_s)}\tilde{L}_{d,2}e^{F(\bar{\psi}_s,\bar{m}_s)} \right\}.
	 \end{align*} 
	 \end{proof}

 One way to choose $a,b,c,d$ is to ensure equations 
 \begin{equation} \label{eq: coupling a'b'c'd'}
 	B = -p_2 (\exp(c)-1))= - q_1(\exp(d)-1)), \quad 
 	b = -d, \quad 
 	c = -a -d, 
 \end{equation}
 under which the time integral in \eqref{eq: exponential martingale} is linear in $t\wedge\tau$, and
 \begin{align} \label{eq: simplified exponential martingale}
 M_t(a,b,c,d) =& \exp\left[ a \left(N_{w,b}(t)-N_{b,p}(t)\right)  + b\left(N_{b,p}(t) +N_{w,p}(t)-N_{p,D}(t)\right) \right. 
 \notag \\ 
 & -a N_{w,b}(0) -b N_{w,p}(0) \notag \\ 
 &+ \left. \int_{0}^{t\wedge\tau}-(p_1+p_2)\left(\rho\exp(a)+(1-\rho)\exp(b)-1\right)\right. 
 \notag\\&\quad 
  \left.-p_2\left(\left(\rho\exp(a)+(1-\rho)\exp(b)\right)^2-1\right) + B\,ds \right].
 \end{align} 
 Let $T_p$ be an exponential random variable with a parameter $q_1$, $T_b$ be the sum of two independent exponential random variables with parameters $q_1$ and $p_2$. We denote by $M_{T_b}(B)$ and $M_{T_p}(B)$ their moment generating functions \begin{align}\label{eq:moment generating function}
 	M_{T_b}(B)=&\frac{p_2}{p_2-B}\frac{q_1}{q_1-B}\\
 	M_{T_p}(B)=&\frac{q_1}{q_1-B}.
 \end{align}
 The second lemma says that we can choose $a, b, c, d$ in terms of the moment generating functions of $T_b$ and $T_p$, such that \eqref{eq: coupling a'b'c'd'} holds for $B<q_1 \wedge p_2$  in a  neighborhood of 0, given $w=p_1+3p_2 < m$:
 \begin{lemma} \label{lm: moment generating functions for tau}
 	Consider the stopped process $(\bar{\psi}_t,\bar{m}_t)$ with initial condition \eqref{eq: renewal process initial condition} and $w=p_1+3p_2 < m$. There exists an $\epsilon>0$, such that for $B \in (-\epsilon,\epsilon)$, we can uniquely solve equations \eqref{eq: coupling a'b'c'd'}. We further have the moment generating function for $\tau$ for $B \in (-\epsilon,\epsilon)$.
 	\begin{equation} \label{eq: mgf for tau}
 		\mathbb{E}^{\mu_\rho,d,2}\left[\exp\left(g(B)\tau\right) \right] = \rho M_{T_b}(B) + (1-\rho) M_{T_p}(B)
 	\end{equation}
 	where
 	\begin{align}\label{eq: g(b)}
 		g(B) = B -&(p_1+p_2) \left(\rho M_{T_b}(B) + (1-\rho)M_{T_b}(B) -1)\right) \notag\\
 		 - & p_2\left(\left(\rho M_{T_b}(B) + (1-\rho)M_{T_b}(B)\right)^2 -1\right) .
 	\end{align}
 \end{lemma}
 \begin{proof}:
 	Solving \eqref{eq: coupling a'b'c'd'} explicitly, we have for $B<p_2$,
 	\[a= \ln M_{T_b}(B), \quad b = \ln M_{T_p}(B),\quad c =\ln\left(1-\frac{B}{p_2}\right),\quad  d'=-\ln M_{T_p}(B).\] 
 	Therefore, we have exponential martingales from \eqref{eq: simplified exponential martingale}
 	\begin{align}\label{eq: uniform exponential martingale}
 		M_t(B) =& \exp\left[ \ln M_{T_b}(B) \left(N_{w,b}(t)-N_{b,p}(t)\right)  +\ln M_{T_b}(B) \left(N_{b,p}(t) -N_{w,p}(t)-N_{p,D}(t)\right) \right. 
 		\notag \\ 
 		& -\ln M_{T_b}(B) N_{w,b}(0) -\ln M_{T_p}(B) N_{w,p}(0) \left. + g(B)\tau\wedge t \right].
 		\end{align}
 	Notice that $g(B)$ is analytic at $0$, with $g(0)=0$. Its derivative is
 	\begin{equation} \label{eq: critical condition}
 		g'(0) = 1 - w \left(\rho \mathbb{E}(T_b) + (1-\rho)\mathbb{E}(T_p)\right) = 1-\frac{w}{m} >0, 
 	\end{equation}
 	where $w = p_1+ 3p_2$. Therefore, there exists $\epsilon>0$, such that for $B \in (-\epsilon,0)$
 	\[g(B)<0.  \] 
 	
 	On the other hand, since the numbers of blue cups and purple cups are nonnegative, and they are both $0$ at the time $\tau$, we have:  \begin{align*} N_{w,b}(t)-N_{b,p}(t) \geq 0 \\ N_{b,p}(t) +N_{w,p}(t)-N_{p,D}(t) \geq 0
 	\end{align*} equalities both hold for $t\geq \tau$. Therefore, $M_t(B)$ is uniformly bounded in time $t\geq 0$ for $B\in(-\epsilon,0)$. Particularly, $M_{T_b}(B),M_{T_p}(B)\leq 1$ and 
 	\[M_t(B) \leq M_{T_b}(B)^{- N_{w,b}(0)} M_{T_p}(B)^{- N_{w,p}(0)} \leq \left(M_{T_b}(B) M_{T_p}(B)\right)^{-1}. \] 
 	By the Optional-Stopping Theorem, we have the moment generating function of $\tau$ for ${B\in (-\epsilon,0)}$ given by \eqref{eq: mgf for tau}. Again by \eqref{eq: critical condition}, we can extend the equation analytically.
 \end{proof}
 
 Now we can prove Theorem \ref{Thm: CLT for X_t}:
   \\
 \begin{proof} (Theorem \ref{Thm: CLT for X_t}):
 	 We can use the standard central limit theorem for renewal processes. For example, see arguments in the proof of Theorem 1.3 \cite{ASV}. We need to show both $\tau$ and $X_\tau$ have some positive finite exponential moments.
 	
 	 By Lemma \ref{lm: moment generating functions for tau}, for some $c_1>0$, we can solve ${g(B)=c_1}$ analytically and get the moment generating function of $\tau$ to be finite, 
 	\begin{equation} \label{eq: finite moment of tau}
 		\mathbb{E}^{\mu_\rho,d,2}\left[\exp\left(c_1\tau\right) \right] <\infty
 	\end{equation}
 	 To show $X_\tau$ also has finite exponential moments. We notice that $X_t = N_{p,D}(t)$, which is dominated by a Poisson process with rate $q_1$. Therefore, by \eqref{eq: finite moment of tau}, for some $c_2>0$,
 	  \begin{equation} \label{eq: finite moment of Xtau}
 	  \mathbb{E}^{\mu_\rho,d,2}\left[\exp\left(c_2 X_\tau\right) \right] <\infty.
	  \end{equation}
	  \end{proof}
	 
	\textcolor{black}{We end this section by discussing the moment generating function and the condition \eqref {eq: critical condition} for the regenerative structure.}
	 \begin{remark}\label{rm: regenerative structure} \begin{enumerate}
	 		\item In view of the law of large numbers of $X_t$, the condition $w<m$ is also necessary for $\tau$ to be finite a.s. To have this condition, we typically need $q_1$ to be large and $\rho, p_1$ to be small. Indeed, if $q_1$ is large but $\rho>1/3$, we have $m = (q_1^{-1} + \rho p^{-1}_2)^{-1}\approx p_2 \rho^{-1} <3 p_2 < w$. 
	 		
	 		\item For the case when $w\geq m$, the behavior of the fluctuation of $X_t$ is unclear. In fact, we can get a regenerative structure for the first color scheme and consider $\tilde{\tau}$ as the return time when there are only white cups. It has a finite first moment, see Remark \ref{rm: ergodic measure}.  As nonwhite cups move like symmetric random walks when the tagged particle is not moving, and creation of nonwhite cup is local at site 1, we might also expect $\tilde{\tau}$ to have finite second moments in the general case.
	 		
	 		\item We can compute the joint moment generating function of $\tau$, $X_\tau$ and also the speed explicitly with the exponential martingales \eqref{eq: exponential martingale}. For example, by ${X_\tau  = N_{w,b}(\tau)+ N_{w,p}(\tau)}$, ${N_{w,b}(0)+ N_{w,p}(0)=1}$, and taking $a=b,c=d=0$ in \eqref{eq: exponential martingale}, we can get, when $w<m$,
	 		\begin{equation}\label{eq: joint mgf}
	 		\mathbb{E}^{\mu_\rho,d,2}\left[\exp\left(b X_\tau - b- h(b)\cdot \tau\right) \right] =1,\end{equation}
	 		where $h(b) =  (p_2+p_1) \left(\exp(b)-1)\right) 
	 		+   p_2 \left(\exp(2b)-1\right)$. Taking derivatives with respect to $b$, we have
	 		\[\mathbb{E}^{\mu_\rho,d,2}\left[X_\tau - 1- h'(0)\cdot \tau \right] =0.\]
	 		Therefore, the speed of the tagged particle is
	 		\begin{equation*} 
	 		\frac{\mathbb{E}^{\mu_\rho,d,2}\left[X_\tau \right] }{\mathbb{E}^{\mu_\rho,d,2}\left[ \tau \right] } 
	 		= \frac{1}{\mathbb{E}^{\mu_\rho,d,2}\left[\tau \right]}+h'(0)
	 		= \frac{g'(0)}{\rho \mathbb{E}(T_b) + (1-\rho)\mathbb{E}(T_p)} +h'(0) =m,
	 		\end{equation*}
	 		which is the same as Theorem \ref{Thm: LLN for X_t}, when $w<m$.
	 	\end{enumerate}
	 \end{remark} 

\section{large deviation Principles} \label{sec: LDP}
In this section, we will prove the large deviation principle for the displacement $X_t$. The main step is to prove the large deviation principle for various jumps and empirical densities at site 1 jointly. Applying the contraction principle, we can get the desired result. We start from the various jumps and empirical densities at site 1.

\subsection{Various Jumps at Site 1, and Their LDP}
We use the first color scheme, where changes of colors only occur at site 1. \textcolor{black}{The various jumps correspond to changes among three colors at site 1 and empirical densities at site 1 are averaged occupation time of different colors at site 1.} We will see that the LDP for the various jumps and empirical densities is mainly a consequence of perturbations and a law of large numbers, see Lemma \ref{lm:Local LDP at positive good points}. By the graphical construction, the original color process can be constructed via event times which correspond to a collection of independent Poisson processes. Varying the jump rates of some Poisson processes, we can get a perturbed color process, which also has a law of large numbers for the various jumps and empirical densities. This has two consequences. First, an atypical event for various jumps and empirical densities under the original measure becomes typical under the new measure. Second, the Radon-Nikodym derivate between the two measures (in a suitable sense) is a function of various jumps and empirical densities at site 1. By the law of large numbers, this Radon-Nikodym derivate is almost a constant over the atypical event, which allows us to get the rate function. Some technical issues and detailed computation will be explained in the next subsection. 

\textcolor{black}{
We first recall the graphical construction of the auxiliary process under the first color scheme. Initially, all cups at positive sites are white. We change the positions and colors of cups at event times of a collection of independent Poisson processes $\mathcal{C}$, $\mathcal{D}$, and ${(\mathcal{N}_{x,y})_{x>y>0}}$. For perturbations, we will replace the Poisson process $\mathcal{C}$ by three independent Poisson processes $\mathcal{C}_{w,b}$, and $\mathcal{C}_{w,D}$, $\mathcal{C}_{p,D}$ with rates $\lambda_{w,b}$, $\lambda_{w,D}$, $\lambda_{p,D}$, respectively. We will also replace $\mathcal{D}$ by two independent Poisson processes $\mathcal{D}_{b,p}$, $\mathcal{D}_{w,p}$, with rates $\lambda_{b,p}$, $\lambda_{w,D}$, respectively. In the original process, these rates are ${\lambda_{w,b} = q_1 \rho}$, ${\lambda_{w,D} = q_1 (1-\rho)}$, ${\lambda_{p,D} = q_1}$, and ${\lambda_{w,p} = \lambda_{b,p}= p_2}$, while in perturbed process, these rates are some chosen positive numbers. Changes of colors and removals of cups always occur at site 1 and at event times of these Poisson processes: a white cup is colored blue at an event time of $\mathcal{C}_{w,b}$, it is removed and the configuration is shifted at an event time of  $\mathcal{C}_{w,D}$, it is colored purple at an event time $\mathcal{C}_{w,p}$; a blue cup is colored purple at an event time $\mathcal{C}_{b,p}$; a purple cup is removed and the configuration is shifted towards left at an event time $\mathcal{C}_{p,D}$. Change of positions of cups occur at sites $x,y$ at event times of $\mathcal{N}_{x,y}$. Similar to section \ref{sec: Regenerative Structure}, there are five processes corresponding to the total numbers of changes of colors and removal of cups at site 1 by time $t$: $n_{w,b}(t)$, $n_{w,p}(t)$, $n_{w,D}(t)$, $n_{b,p}(t)$, and $n_{p,D}(t)$. 
They have (varying) jump rates ${\lambda_{w,b}\cdot \mathbb{1}_{\{c(1)=w\}}}$, ${\lambda_{w,D}\cdot \mathbb{1}_{\{c(1)=w\}}}$, ${\lambda_{p,D}\cdot \mathbb{1}_{\{c(1)=p\}}}$, ${\lambda_{w,p}\cdot \mathbb{1}_{\{c(1)=w\}}}$, and ${\lambda_{b,p}\cdot \mathbb{1}_{\{c(1)=b\}}}$, respectively. We will also consider the empirical densities of three colors at site 1: ${m_w = \frac{1}{t}\int_{0}^{t}\mathbb{1}_{\{c_s(1)=w\}}\,ds}$, ${m_p = \frac{1}{t}\int_{0}^{t}\mathbb{1}_{\{c_s(1)=p\}}\,ds}$, and ${m_b = \frac{1}{t}\int_{0}^{t}\mathbb{1}_{\{c_s(1)=b\}}\,ds}$. We denote by $\vec{n}$ the tuple ${\left(n_{w,b}, n_{w,D}, n_{p,D}, n_{w,p}, n_{b,p}\right)}$, and by ${\vec{m}}$ the tuple ${\left(m_w,m_p,m_b \right)}$. Now we can describe the rate function for $\left(\frac{1}{t}\vec{n},\vec{m}\right)$.}

{

We say a point $x = (x_i) \in \mathbb{R}^8$ is "good" if it satisfies equations:
\begin{equation} \label{eq: good point}
x_1 -x_5 =0, -x_3+x_4+x_5 =0 ,  x_6+x_7+x_8=1,
\end{equation} which are analogues of \eqref{eq: marginal equations}. Let $ H(\cdot|\cdot)$ be a function from $\mathbb{R}^2$ to $\mathbb{R}\bigcup\{\infty\}$, defined as 
\begin{equation}
H(s|t) = \begin{cases}
s \ln s - s\ln t - s+t & \quad  \text{if }\  s,t \geq 0\\
+\infty & \quad \text{otherwise}
\end{cases},
\end{equation} with the convention that $0\cdot \ln 0 = 0,\ln 0 = -\infty$. The rate function $\mathbb{J}(\cdot)$ can be defined on $\mathbb{R}^8$ as
\begin{equation} \label{eq: rate function for joint}
\mathbb{J}(x) = \begin{cases}
H(x_1|\lambda_{w,b}\cdot x_6)+ H(x_2|\lambda_{w,D}\cdot x_6) &\quad \text{if } x \text{ is "good" } \\
\quad \quad+ H(x_3|\lambda_{p,D}\cdot x_7) + H(x_4|\lambda_{w,p}\cdot x_6)   \\
\quad \quad+ H(x_5|\lambda_{b,p}\cdot x_8) \\
+\infty & \quad \text{otherwise}
\end{cases}, 
\end{equation}
where $\lambda_{w,b} = q_1 \rho$, $\lambda_{w,D} = q_1 (1-\rho)$, $\lambda_{p,D} = q_1 $, $\lambda_{w,p} = \lambda_{b,p}= p_2$. It is easy to see functions $\mathbb{J}(\cdot)$ is a good rate function, which is a function with compact sublevel sets.

Now we state the main theorem of this section, the large deviation principle for $\left(\frac{1}{t}\vec{n},\vec{m}\right)$.
\begin{theorem} (Annealed LDP for Varaious Jumps and Empirical Densities at Site 1) \label{thm: LDP for joint}\\
	Under the assumption of Theorem \ref{Thm: LLN for X_t}, the various jumps and empirical densities at site 1 satisfy the LDP with the good rate function $\mathbb{J}(\cdot)$. For any Borel set $B \subset \mathbb{R}^8$,
	\begin{align} -\inf_{x\in B^o} \mathbb{J}(x) &\leq \liminf_{t\to \infty}  \frac{1}{t} \ln \mathbb{Q}^{\bar{\mu}_\rho,d,1}\left((\frac{1}{t}\vec{n},\vec{m}) \in B \right)  \notag \\
	&\leq \limsup_{t\to \infty}  \frac{1}{t} \ln \mathbb{Q}^{\bar{\mu}_\rho,d,1}\left((\frac{1}{t}\vec{n},\vec{m}) \in B \right) \leq -\inf_{x\in \overline{B}} \mathbb{J}(x) .
	\end{align}  
\end{theorem}

Since the displacement  $X_t$ of the tagged particle can be written as 
\begin{equation}\label{eq:ez relation}
X_t = n_{w,D}(t) + n_{p,D}(t). 
\end{equation} 
the displacement $X_t$ also satisfies an LDP with a rate function $\mathbb{I}(\cdot) : \mathbb{R} \to \mathbb{R}\bigcup\{\infty\} $ as
\begin{equation} \label{eq: rf from contraction principle}
\mathbb{I}(y) = \inf \{ \mathbb{J}(x): x_2 + x_3 = y \},
\end{equation}which is an application of the contraction principle. 
\begin{corollary}
	(Theorem \ref{thm: LDP}, Annealed LDP for the displacement $X_t$) \label{thm: LDP for displacemt}\\
	Under the assumption of Theorem \ref{Thm: LLN for X_t}, the displacement $X_t$ has the LDP with a good rate function $\mathbb{I}(\cdot)$. For any Borel set $C \subset \mathbb{R}$,
	\begin{align} -\inf_{y\in C^o} \mathbb{I}(y) &\leq \liminf_{y\to \infty}  \frac{1}{t} \ln \mathbb{Q}^{\bar{\mu}_\rho,d,1}\left(\frac{X_t}{t} \in C \right) \notag \\
	&\leq \limsup_{t\to \infty}  \frac{1}{t} \ln \mathbb{Q}^{\bar{\mu}_\rho,d,1}\left(\frac{X}{t} \in C \right) \leq -\inf_{y\in \overline{C}} \mathbb{I}(y) .
	\end{align}  
\end{corollary} 
\begin{proof}:
	Apply the contraction principle to equation \eqref{eq:ez relation}. See Theorem 4.2.1 \cite{DZ}.
\end{proof}

\textcolor{black}{
Before the proof of the LDP for various jumps, we briefly explain the rate function on "good" points.
\begin{remark} 
	The rate function on "good points" can be seen from the proof of Lemma \ref{lm:Local LDP at positive good points}, especially from the Radon-Nikodym derivative, \eqref{eq: r-n derivative}. It can be interpreted as the cost of perturbing the five various jumps jointly. The cost is a sum because they are disjoint, and each term is similar to that of Poisson processes, as the relative entropy. In principle, to change a various jump, which is time-inhomogeneous, we change two variables, a new rate $x_i = \tilde{\lambda}_i\cdot \mu_i$ and a reference rate $\lambda_i \mu_i$, where $\mu_i$ is some empirical density. In this case, we only need to change eight variables because all $\mu_i$ are empirical densities of colors at site 1, $\mu_i= x_j$ for some $j=6,7,8$. On the other hand, the law of large numbers in section \ref{sec: LLN for displacement} provides three extra equations, \eqref{eq: good point}. Thus, we have five degrees of freedom, and we can choose $\tilde{\lambda}$ for different $x$. A very similar phenomena can also be  observed in \cite{AJV}, where the joint LDP has the same form as the joint LDP of independent Poisson processes.
\end{remark}
}

\subsection{Three Lemmas for the LDP of Various Jumps}
We will divide the proof of Theorem \ref{thm: LDP for joint} into the following three lemmas. \textcolor{black}{These lemmas are standard.}

The first lemma is exponential tightness, it follows directly from $\vec{n}_i$ are dominated by Poisson processes, and empirical densities $m_i$ are bounded. 
\begin{lemma} (Exponential Tightness) \label{lm: exponential tightness}
	The distributions of $\left(\frac{1}{t}\vec{n},\vec{m}\right)$ are	exponentially tight. That is, for any $l>0$, exists a compact set $K_l \subset \mathbb{R}^8$, such that, for any closed set $C \subset K_l^c$,
	\begin{equation}
	\limsup_{t\to\infty} \frac{1}{t} \ln \mathbb{Q}^{\bar{\mu}_\rho,d,1}\left( (\frac{1}{t}\vec{n},\vec{m}) \in C \right) \leq -l.
	\end{equation} 
\end{lemma}
\begin{proof}:
	Since $\vec{n}_i$ are (stochastically) dominated by orthogonal Poisson processes with bounded rates, and the empirical densities are bounded by 1. We have 
	\[\limsup_{M\to \infty}\limsup_{t\to\infty} \frac{1}{t} \ln \mathbb{Q}^{\bar{\mu}_\rho,d,1}\left( \frac{1}{t}\vec{n}_i \geq M,\vec{m}_j\geq 1, \text{ for all i,j}  \right) = -\infty, \]
	which is sufficient.
\end{proof}

The second lemma says that the number of nonwhite cups has a sub-linear bound in time $t$. This lemma allows us to consider only the cases when $n_{w,p} + n_{b,p} \approx n_{p,D}$ and $n_{w,b} \approx n_{b,p}$, which correspond to $x$ are "good". The proof relies on exponential martingales and Chebyshev Inequality.
\begin{lemma}(Sub-linear Bound of Colored Cups) \label{lm:sublinear growh}
	Let $N_t$ be the (net) number of nonwhite cups at time $t$, that is $N_t= N_b(t) + N_p(t)$.  For any $\delta>0$, we have
	\begin{equation} \label{eq:key estimate for good points}
	\limsup_{t\to\infty} \frac{1}{t} \ln \mathbb{Q}^{\bar{\mu}_\rho,d,1}\left( \frac{N_t}{t} \geq \delta  \right) = -\infty.
	\end{equation}
\end{lemma}
\begin{proof}:
	We shall consider the weight function $\hat{W} = \sum_{i>0} i\cdot \mathbb{1}_{\{c(i)\neq w\}}$ from section \ref{sec: LLN for displacement}, which can be written as $\hat{W}=W_b+W_p$. Clearly, $\hat{W}(0)=0$ initially. There is a class of $\mathbb{Q}^{\bar{\mu}_\rho,d,1}$-exponential martingales, 
	\[M_t(a)= \exp\left(a\hat{W}(t)- \int_0^t e^{-a\hat{W}(s)}\tilde{L}_{d,1}e^{a\hat{W}(s)}\,ds \right). \]
	Therefore, by Cauchy-Swartz Inequality,
	\begin{align}
	\mathbb{E}^{\bar{\mu}_\rho,d,1}\left[\exp\left(a\hat{W}(t)\right) \right] =& \mathbb{E}^{\bar{\mu}_\rho,d,1}\left[M_t(2a)^{\frac{1}{2}} \exp\left( \frac{1}{2}\int_0^t e^{-2a\hat{W}(s)}\tilde{L}_{d,1}e^{2a\hat{W}(s)}\,ds\right) \right] \notag \\	
	\leq& \mathbb{E}^{\bar{\mu}_\rho,d,1}\left[ \exp\left(
	\int_0^t e^{-2a\hat{W}(s)}\tilde{L}_{d,1}e^{2a\hat{W}(s)}\,ds \right) \right]^{\frac{1}{2}}. \label{eq: C-S estimate}	\end{align}
	
	To estimate the last term, we notice $e^{-2a\hat{W}(s)}\tilde{L}_{d,1}e^{2a\hat{W}(s)}$ can be written into two pieces. The first piece corresponds to the symmetric exclusion, which is of the form
	\[ p_2 \left(e^{2a}-e^{-2a}\right)^2 B_2(s) + p_1\left(e^a-e^{-a}\right)^2 B_1(s), \] where $B_2(s), B_1(s)$ are bounded above by $N_s$
	. The second term corresponds to the boundary effects due to shifts of the configurations and interchanges near site 1. It is bounded above by
	\[c(q_1+p_1+p_2)(e^a-1)t\] for some constant $c>0$. Combining these two terms, and notice $N_s \leq n_{w,p}(s) + n_{w,b}(s)$, which is (stochastically) dominated by some Poisson process with a bounded rate, we have
	\begin{align}
	\mathbb{E}^{\bar{\mu}_\rho,d,1}\left[\exp\left(a\hat{W}(t)\right) \right] \leq& \exp(c(e^a-1)t) \mathbb{E}^{\bar{\mu}_\rho,d,1}\left[ \exp\left(
	\int_0^t A(a) N_s \,ds \right) \right]^{\frac{1}{2}} \notag \\
	\leq& \exp \left(c(e^a-1)t + c(e^{A(a)t}-1)t \right) \label{eq: C-S, P}
	\end{align}
	where $A(a) =  p_2 \left(e^{2a}-e^{-2a}\right)^2 +  p_1 \left(e^{a}-e^{-a}\right)^2$, which behaves quadratically when $a$ is small.
	
	Since $N_t \geq \delta t$ implies $\hat{W}_t \geq \frac{1}{10}(\delta t)^2$, by \eqref{eq: C-S estimate}, \eqref{eq: C-S, P} we can get
	\begin{align*}
	\mathbb{Q}^{\bar{\mu}_\rho,d,1}\left( \frac{N_t}{t} \geq \delta  \right)\leq& \mathbb{Q}^{\bar{\mu}_\rho,d,1}\left( \hat{W}_t \geq \frac{1}{10}(\delta t)^2  \right)	\notag \\
	\leq & \mathbb{E}^{\bar{\mu}_\rho,d,1}\left[\exp\left(a\hat{W}(t)\right) \right] \exp(-\frac{1}{10} \cdot a (\delta t)^2) \notag \\
	\leq & \exp\left(c(e^a-1)t + c(e^{A(a)t}-1)t -\frac{1}{10} \cdot a (\delta t)^2 \right).
	\end{align*} Choosing $a = t^{-\frac{1}{2}}$, and using $A(a)$ behaves quadratically near 0, we have, for large $t$,
	\begin{equation}
	\frac{1}{t}\ln\mathbb{Q}^{\bar{\mu}_\rho,d,1}\left( \frac{N_t}{t} \geq \delta  \right) \leq -ct^{\frac{1}{2}},
	\end{equation} for some $c>0$ depending on $\delta$. This is sufficient for the estimate \eqref{eq:key estimate for good points}.

\end{proof}

The last lemma provides a local estimate for the LDP rate function on "good" points $x$ with positive entries. This is a consequence of the Girsanov-type formula for the Radon-Nikodym derivatives, and the ergodic theorem at the end of section \ref{sec: LLN for displacement}. Particularly, we will see $x$ and $\tilde{\lambda}$ are related via equations \eqref{eq: parameter equation} and \eqref{eq: good point}.
\begin{lemma}(Local LDP Estimates) \label{lm:Local LDP at positive good points}
	Let $B(x,r)$ be a ball with center $x$ and radius $r$ in $\mathbb{R}^8$, with respect to the sup-norm. Suppose $x$ is a "good" point with positive entries, that is, $x$ satisfies the condition \eqref{eq: good point}, and $x_i>0$ for all $i$,  then we have
	\begin{equation} \label{eq: local estimate}
	\limsup_{r\downarrow 0}\limsup_{t\to\infty} \abs{\frac{1}{t} \ln \mathbb{Q}^{\bar{\mu}_\rho,d,1}\left(( \frac{1}{t}\vec{n},\vec{m}) \in B(x,r)  \right)+ \mathbb{J}(x) }= 0.
	\end{equation} 
\end{lemma} 
\begin{proof}:
	\textcolor{black}{In view of the graphical construction of the auxiliary process described at the beginning of this section, we will modify the jump rates of independent Poisson processes $\mathcal{C}_{w,b}$, $\mathcal{C}_{w,D}$ ,$\mathcal{C}_{p,D}$, $\mathcal{C}_{w,p}$ and $\mathcal{C}_{b,p}$ from  ${\lambda=(\lambda_{w,b},\lambda_{w,D}, \lambda_{p,D},\lambda_{w,p},\lambda_{b,p})}$ to ${\tilde{\lambda}= (\tilde{\lambda}_{w,b},\tilde{\lambda}_{w,D},\tilde{\lambda}_{p,D},\tilde{\lambda}_{w,p},\tilde{\lambda}_{b,p})}$. This allows us to get a perturbed color process. We denote by $\mathbb{Q}^{\tilde{\lambda},d,3}$ the corresponding probability measure on the space of c\'{a}dl\'{a}g paths on $\tilde{\mathbb{X}} =\{b,p,w\}^{\mathbb{Z}_+}$ when initial cups are all white. The perturbed process has a generator $\tilde{L}_{d,3}$, which acts on a local function $f$ by}
	\begin{equation*} \label{eq: L_d,3}
	\tilde{L}_{d,3}f(\eta) =\left(\tilde{S}^{ex}_{+,1} + \tilde{L}^{sh,3} + \tilde{L}^{d,3}\right)f(\eta), \end{equation*} 
	where
	\begin{align}
	\tilde{L}^{sh,3} f(\eta)=& \tilde{\lambda}_{w,D}\cdot \mathbb{1}_{\{c(1)=w\}}\left(f(\theta_1\circ C_{p,1}\eta)-f(\eta)\right) \notag \\
	&+ \tilde{\lambda}_{w,b}\cdot \mathbb{1}_{\{c(1)=w\}}\left(f(C_{b,1}\eta)-f(\eta)\right) \notag \\
	&+ \tilde{\lambda}_{p,D} \mathbb{1}_{\{c(1)=p\}}\left(f(\theta_1\eta)-f(\eta)\right) \notag\\		
	\tilde{L}^{d,3}f(\eta) =& \tilde{\lambda}_{w,p}\cdot \mathbb{1}_{\{c(1)=w\}} \left(f(C_{p,1}\eta)-f(\eta)\right) \notag \\
	&+\tilde{\lambda}_{b,p}\cdot \mathbb{1}_{\{c(1)=b\}} \left(f(C_{p,1}\eta)-f(\eta)\right). \notag
	\end{align}
	It is clear that when $\tilde{\lambda} = \lambda$, we recover the original process and its generator $\tilde{L}_{d,1}$. 
	
	When every jump rates $\tilde{\lambda}_i$ are positive, the measures of the two color processes (with $\sigma$-algebras generated by $(\eta_s)_{s\leq t}$) are absolutely continuous with respect to each other up to time $t$. The Radon-Nikodym derivative ${\Psi_t(\tilde{\lambda}) =\frac{d\mathbb{Q}^{\lambda,d,3} }{d\mathbb{Q}^{\tilde{\lambda},d,3}}}$ is
	
	\begin{align}\label{eq: r-n derivative}
	\Psi_t(\tilde{\lambda}) = \exp& \left(  \sum_{i=1,2,4} ( \ln \lambda_i -\ln \tilde{\lambda}_i )n_i - \int_0^t \sum_{i=1,2,4} (\lambda_i-\tilde{\lambda}_i)\cdot \mathbb{1}_{\{c_s(1)=w\}}\,ds 
	\right.
	\notag \\
	& \left.
	+ ( \ln \lambda_3 -\ln \tilde{\lambda}_3 )n_3 - \int_0^t (\lambda_3-\tilde{\lambda}_3)\cdot \mathbb{1}_{\{c_s(1)=p\}}\,ds   \right.			
	\\
	& \left. + (\ln \lambda_5 -\ln \tilde{\lambda}_5 )n_5 - \int_0^t (\lambda_5 - \tilde{\lambda}_5 )\cdot \mathbb{1}_{\{c_s(1)=p\}}\,ds \right). \notag
	\end{align}
	This is a function of various jumps and empirical densities at site 1. As a consequence, if we denote by $A_{x,r,t}$ the event  $(\frac{1}{t}\vec{n},\vec{m}) \in B(x,r)$ , we have the following estimate for $\Psi_t(\tilde{\lambda})$ on $A_{x,r,t}$, up to an error of size $r\cdot t$ (inside $\exp$),
	\begin{align} \label{eq: local estimate 2}
	\Psi_t(\tilde{\lambda}) = \exp & -\left\lgroup  \sum_{i=1,2,4} ( \ln \tilde{\lambda}_i - \ln \lambda_i )x_i - \sum_{i=1,2,4} (\tilde{\lambda}_i - \lambda_i)x_6 
	\right.
	\notag \\
	& \left.
	+ ( \ln \tilde{\lambda}_3 - \ln \lambda_3 )x_3 -  (\tilde{\lambda}_3 - \lambda_3)x_7   \right.			
	\\
	& \left. + ( \ln \tilde{\lambda}_5 - \ln \lambda_5 )x_5 - (\tilde{\lambda}_5 - \lambda_5)x_8 \right\rgroup \cdot t \notag
	\end{align}
	On the other hand, we can always choose positive $\{\tilde{\lambda}_i\}$ to solve equations 
	\begin{align}
	\tilde{\lambda}_i \cdot x_6 &= x_i, i =1,2,4,\notag \\
	\tilde{\lambda}_3 \cdot x_7 &= x_3, \label{eq: parameter equation}\\
	\tilde{\lambda}_5 \cdot x_8 &= x_5, \notag
	\end{align} when $x$ is a "good" point with positive entries. \textcolor{black}{Given this particular $\tilde{\lambda}$, we can show a law of large numbers for 
		$\mathbb{Q}^{\tilde{\lambda},d,3}$ as a consequence of positive $\{\tilde{\lambda}_i\}$. Following the proof in section \ref{sec: LLN for displacement}, we need to solve $\bar{\nu}_w$, $\bar{\nu}_p$, $\bar{\nu}_b$ from \eqref{eq: marginal equations} in terms of $\tilde{\lambda}$. In view of \eqref{eq: good point} and \eqref{eq: parameter equation}, $\bar{\nu}_w = x_6, \bar{\nu}_p= x_7, \bar{\nu}_b= x_8$ is the unique solution to \eqref{eq: marginal equations}. By equation \eqref{eq: ergodic theorem}, and analogues of \eqref{eq: LLN for jumps}, we get a law of large numbers for various jumps and empirical densities at site 1 and} 
	\begin{equation} \label{eq: general LLN}
	\lim_{t\to\infty}\mathbb{Q}^{\tilde{\lambda},d,3}\left(A_{x,r,t}\right) = 1.  \end{equation}
	 Furthermore, equation \eqref{eq: local estimate 2} becomes
	\begin{equation}
	\Psi_t(\tilde{\lambda}) = \exp  -t \left( \mathbb{J}(x) + O(r) \right)
	\end{equation}
	Therefore, we have
	\begin{align}
	\frac{1}{t}\ln \mathbb{Q}^{\lambda,d,3}\left(A_{x,r,t}\right) =& \frac{1}{t}\ln \int_{A_{x,r,t}}  \Psi_t(\tilde{\lambda}) \,d \mathbb{Q}^{\tilde{\lambda},d,3} \notag \\
	=&  \mathbb{J}(x) + O(r) + \frac{1}{t}\ln \mathbb{Q}^{\tilde{\lambda},d,3}\left(A_{x,r,t}\right)  ,
	\end{align} which is sufficient for estimate \eqref{eq: local estimate} from \eqref{eq: general LLN}.
\end{proof}

\subsection{Proof of the LDP of Various Jumps}
Now we can prove Theorem \ref{thm: LDP for displacemt}.

\begin{proof}: 
	We will divide the proof into three steps.
	\begin{enumerate}[label = S\arabic*.]
		\item Local LDP estimates for good points with positive entries:
		
		Notice that by any time $t$, the (net) numbers of revealed black cups and revealed purple cups are positive:
		\begin{align} \label{eq: positive growth}
		N_b= n_{w,b} -n_{b,D} \geq 0, \notag \\
		N_p = n_{w,p}+n_{b,p}-n_{p,D} \geq 0.  \end{align} 
		By Lemma \ref{lm:sublinear growh}, we only need to consider the asymptotic behavior of 
		\[\frac{1}{t} \ln \mathbb{Q}^{\bar{\mu}_\rho,d,1}\left(A_{x,r,t}  \right)=\frac{1}{t} \ln \mathbb{Q}^{\bar{\mu}_\rho,d,1}\left(( \frac{1}{t}\vec{n},\vec{m}) \in B(x,r)  \right)\] for "good" points $x$ with nonnegative entries. We want to show
		\begin{align} \label{eq: local LDP1}
		\lim_{r \downarrow 0} \limsup_{t\to\infty}\frac{1}{t} \ln \mathbb{Q}^{\bar{\mu}_\rho,d,1}\left(A_{x,r,t}  \right)
		\leq & -\mathbb{J}(x) \\
		\lim_{r \downarrow 0} \liminf_{t\to\infty}\frac{1}{t} \ln \mathbb{Q}^{\bar{\mu}_\rho,d,1}\left(A_{x,r,t}  \right)
		\geq & -\mathbb{J}(x) \label{eq: local LDP2}
		\end{align}
		
		The case for $x$ with positive entries is covered by Lemma \ref{lm:Local LDP at positive good points}. We only need to consider the case when some $x_i =0$. 
		
		\item Local LDP estimates for good points with some zero entries:
		
		The lower bound \eqref{eq: local LDP2} for a good point $x$ with some zero entries can be obtained by using a sequence of good points $\{x_n\}$ with positive entries. 
		
		\begin{align}
		\lim_{n\to\infty} x_n &= x,\\
		\lim_{n\to\infty} \mathbb{J}(x_n) &= \mathbb{J}(x). 
		\end{align} The existence of such a sequence $\{x_n\}$ is due to the regularity of $H(\cdot|\cdot)$ on its domain  and its boundary,
		\[H(s|t) = \liminf_{(a,b) \to (s,t)} H(a|b), \text{ for } s,t \geq 0. \] 
		Therefore, we can get, as $r_n \downarrow 0$,
		\begin{align*} &\lim_{r \downarrow 0} \liminf_{t\to\infty}\frac{1}{t} \ln \mathbb{Q}^{\bar{\mu}_\rho,d,1}\left(A_{x,r,t}\right) \\
		\geq &\lim_{n\to \infty} \liminf_{t\to\infty}\frac{1}{t} \ln \mathbb{Q}^{\bar{\mu}_\rho,d,1}\left(A_{x_n,r_n,t}\right)\\
		\geq &\lim_{n\to \infty}  -\mathbb{J}(x_n)  = -\mathbb{J}(x).
		\end{align*}
		
		On the other hand, the upper bound \eqref{eq: local LDP1} follows from an argument similar to the proof of Lemma \ref{lm:Local LDP at positive good points}. By estimate \eqref{eq: local estimate 2}, we have, for positive $\tilde{\lambda}_i >0$,
		\begin{align*}
		\mathbb{Q}^{\bar{\mu}_\rho,d,1}\left(A_{x,r,t}\right) \leq \sup_{A_{x,r,t}}& \Psi_t(\tilde{\lambda})\\
		= \exp & -\left\lgroup  \sum_{i=1,2,4} ( \ln \tilde{\lambda}_i - \ln \lambda_i )x_i - \sum_{i=1,2,4} (\tilde{\lambda}_i - \lambda_i)x_6 
		\right.
		\notag \\
		& \left.
		+ ( \ln \tilde{\lambda}_3 - \ln \lambda_3 )x_3 -  (\tilde{\lambda}_3 - \lambda_3)x_7   \right.			
		\\
		& \left. + ( \ln \tilde{\lambda}_5 - \ln \lambda_5 )x_5 - (\tilde{\lambda}_5 - \lambda_5)x_8 +O(r) \right\rgroup \cdot t.  \notag
		\end{align*} Then, for positive $\tilde{\lambda}_i >0$, we can get,
		\begin{align*}
		\lim_{r \downarrow 0} \limsup_{t\to\infty}\frac{1}{t} \ln \mathbb{Q}^{\bar{\mu}_\rho,d,1}\left(A_{x,r,t}\right) \leq& -\left\lgroup \sum_{i=1,2,4} ( \ln \tilde{\lambda}_i - \ln \lambda_i )x_i - \sum_{i=1,2,4} (\tilde{\lambda}_i - \lambda_i)x_6 \right.
		\notag \\
		&
		+ ( \ln \tilde{\lambda}_3 - \ln \lambda_3 )x_3 -  (\tilde{\lambda}_3 - \lambda_3)x_7  			
		\\
		& \left. + ( \ln \tilde{\lambda}_5 - \ln \lambda_5 )x_5 - (\tilde{\lambda}_5 - \lambda_5)x_8 \right\rgroup
		\end{align*} Minimizing over $\tilde{\lambda}_i$, we get the upper bound \eqref{eq: local LDP1} via the explicit formula \eqref{eq: rate function for joint}.
		
		\item Full LDP:
		
		The extension from local LDP estimate to the full LDP is standard, see Theorem 4.1.11 \cite{DZ}.	By the local LDP estimates \eqref{eq: local LDP1}, \eqref{eq: local LDP2} and Lemma \ref{lm: exponential tightness}, we have the full large deviation principle for various jumps and empirical densities at site 1.
	\end{enumerate}
\end{proof}

 \end{document}